\documentclass[a4paper]{article}
\usepackage{mathrsfs}
\usepackage{enumitem}
\usepackage{amssymb}
\usepackage{amsmath}
\usepackage{pict2e}
\usepackage{hyperref}
\usepackage{amsmath}
\usepackage{amsfonts}
\usepackage{amsthm}
\usepackage{graphicx}
\usepackage[small]{caption}
\usepackage{subfigure}

\def\0{\emptyset}

\def\GG{\mathcal{G}}
\newcommand{\ch}{\mathrm{ch}}

\begin{document}
\newtheorem{theorem}{Theorem}
\newtheorem{corollary}[theorem]{Corollary}
\newtheorem{definition}[theorem]{Definition}
\newtheorem{conjecture}[theorem]{Conjecture}
\newtheorem{problem}[theorem]{Problem}
\newtheorem{lemma}[theorem]{Lemma}
\newtheorem{observation}[theorem]{Observation}
\newtheorem{proposition}[theorem]{Proposition}
\newcommand{\remark}{\medskip\par\noindent {\bf Remark.~~}}
\newcommand{\pp}{{\it p.}}
\newcommand{\de}{\em}

\title{Fractional coloring of planar graphs of girth five}
\author{Zden\v{e}k Dvo\v{r}\'{a}k\thanks{Computer Science Institute (CSI) of Charles University,
           Malostransk\'{e} n\'am\v{e}st\'{\i} 25, 118 00 Prague,
	              Czech Republic. E-mail: \protect\href{mailto:rakdver@iuuk.mff.cuni.cz}{\protect\nolinkurl{rakdver@iuuk.mff.cuni.cz}}. Supported by project 17-
04611S (Ramsey-like aspects of graph coloring) of Czech Science Foundation.}\and Xiaolan Hu\thanks{School of Mathematics and Statistics \& Hubei Key Laboratory of Mathematical Sciences, Central China Normal University,
Wuhan 430079, PR China. Partially supported by NSFC
under grant number 11601176 and NSF of Hubei Province under grant number 2016CFB146.}
}
\date{}
\maketitle

\begin{abstract}
A graph $G$ is \emph{$(a:b)$-colorable} if there exists an assignment of $b$-element subsets of $\{1,\ldots,a\}$ to vertices of $G$ such that
sets assigned to adjacent vertices are disjoint. We first show that for every triangle-free planar graph $G$ and a vertex $x\in V(G)$,
the graph $G$ has a set coloring $\varphi$ by subsets of $\{1,\ldots,6\}$ such that $|\varphi(v)|\geq 2$ for $v\in V(G)$ and $|\varphi(x)|=3$.
As a corollary, every triangle-free planar graph on $n$ vertices is $(6n:2n+1)$-colorable. We further use this result to prove
that for every $\Delta$, there exists a constant $M_{\Delta}$ such that every planar graph $G$ of girth at least five and maximum degree $\Delta$
is $(6M_{\Delta}:2M_{\Delta}+1)$-colorable.  Consequently, planar graphs of girth at least five with bounded maximum degree $\Delta$
have fractional chromatic number at most $3-\frac{3}{2M_{\Delta}+1}$.

\end{abstract}

\section{Introduction}

A function that assigns sets to all vertices of a graph is a \emph{set coloring} if the sets
assigned to adjacent vertices are disjoint.
For positive integers $a$ and $b\le a$, an {\em $(a:b)$-coloring} of a graph $G$ is a set coloring
with range $\binom {\{1,\ldots, a\}}{b}$, i.e., a set coloring that to each vertex assigns a $b$-element
subset of $\{1,\ldots, a\}$.
The concept of $(a:b)$-coloring is a generalization of the conventional vertex coloring. In fact,
an $(a:1)$-coloring is exactly an ordinary proper $a$-coloring.
The {\em fractional chromatic number} of $G$, denoted by $\chi_f(G)$, is the infimum of the
fractions $a/b$ such that $G$ admits an $(a:b)$-coloring. Note that $\chi_f(G)\leq \chi(G)$ for any graph $G$,
where $\chi(G)$ is the chromatic number of $G$.

Much of the interest in the chromatic properties of triangle-free planar graphs stems from
Gr\"{o}tzsch's theorem~\cite{grotzsch1959}, stating that such graphs are 3-colorable.
Even in the fractional coloring setting, it is not possible to significantly improve Gr\"{o}tzsch's theorem.
For any positive integer $n$ such that $n\equiv 2 \pmod 3$, Jones~\cite{Jon84} constructed a
triangle-free planar graph on $n$ vertices with independence number
$\frac{n+1}{3}$.  Since $\alpha(G)\ge |V(G)|/\chi_f(G)$, these graphs have
fractional chromatic number at least $\frac{3n}{n+1}=3-\frac{3}{n+1}$ (in fact, they are $(3n:n+1)$-colorable).
Thus, there exist triangle-free planar graphs with fractional chromatic number arbitrarily close to 3.
On the other hand, Dvo\v{r}\'{a}k, Sereni and Volec~\cite{frpltr} showed that there does not exist a triangle-free planar graph with fractional
chromatic number exactly 3 by establishing the following upper bound.

\begin{theorem}[Dvo\v{r}\'{a}k, Sereni and Volec~\cite{frpltr}]\label{(9n:3n+1)}
Every planar triangle-free graph on $n$ vertices is $(9n:3n+1)$-colorable, and thus it has fractional
chromatic number at most $3-\frac{3}{3n+1}$.
\end{theorem}

Note that the graphs built by Jones~\cite{Jon84} contain a large number
of separating 4-cycles.  Motivated by this observation,
Dvo\v{r}\'{a}k, Sereni and Volec~\cite{frpltr} conjectured
that triangle-free plane graphs without separating 4-cycles cannot have fractional chromatic
number arbitrarily close to $3$, and proved this is the case under an additional assumption that the
maximum degree is at most $4$.  They also remarked that since faces of length four are usually easy to deal
with in the proofs by collapsing, a key step would be to prove this conjecture for planar graphs of girth at least five
(this special case was previously conjectured by Dvo\v{r}\'ak and Mnich~\cite{dmnichfull}).

\begin{conjecture}\label{conj-main}
There exists a real number $c<3$ such that every planar graph of girth at least five has fractional chromatic number at most $c$.
\end{conjecture}

The purpose of this work is to establish the following upper bound on the
fractional chromatic number of planar graphs of girth at least five with maximum
degree $\Delta$, proving Conjecture~\ref{conj-main} for graphs with bounded maximum degree.

\begin{theorem}
\label{frac-girth5}
For every positive integer $\Delta$, there exists a positive integer $M_{\Delta}$ as follows.
If $G$ is a planar graph of girth at least five and maximum degree at most $\Delta$, then $G$ is $(6M_\Delta:2M_\Delta+1)$-colorable,
and thus $\chi_f(G)\leq 3-\frac{3}{2M_{\Delta}+1}$.
\end{theorem}

Theorem~\ref{frac-girth5} is an easy corollary of the following result on special set colorings of planar graph of girth at least five.

\begin{theorem}\label{X-girth5}
For every positive integer $k$, there exists a positive integer $s$ such that the following holds.
Let $G$ be a planar graph of girth at least five and let $X$ be a set of vertices of $G$ of degree at most $k$.
If the distance between vertices of $X$ is at least $s$, then $G$ has a set coloring $\varphi$ by subsets of
$\{1,\ldots,6\}$ such that $|\varphi(v)|\geq 2$ for $v\in V(G)$ and $|\varphi(x)|=3$ for $x\in X$.
\end{theorem}

Using standard techniques, we can argue that it suffices to prove Theorem~\ref{X-girth5} in the special case $|X|=1$.
In this special case, we only need to assume that the graph is triangle-free (rather than having girth at least five).

\begin{theorem}
\label{x-tri-free}
Let $G$ be a triangle-free planar graph. For any vertex $x\in V(G)$, the graph $G$ has a set
coloring $\varphi$ by subsets of $\{1,\ldots,6\}$ such that
$|\varphi(v)|\geq 2$ for $v\in V(G)$ and $|\varphi(x)|=3$.
\end{theorem}

Let us remark that in Theorem~\ref{X-girth5}, it does not suffice to forbid triangles: It is easy to see that any graph $G$
satisfying the outcome of the theorem has an independent set of size at least $\frac{2n+|X|}{6}$, implying that
for the graphs constructed by Jones~\cite{Jon84} (which have unbounded diameter), the outcome cannot be true for any set $X$ of size at least three.
It might be possible to improve the ratio of extra colors assigned to the vertex $x$ in Theorem~\ref{x-tri-free} a bit;
e.g., it could be true that there exists a coloring by subsets of $\{1,\ldots,9\}$ such that all vertices get at least three colors
and $x$ gets five.  However, when $G$ is the graph obtained from the wheel with five spokes by subdividing each of the spokes once,
$x$ is the center of the wheel, and $\varphi$ is a coloring by subsets of $\{1,\ldots,k\}$ and each vertex has
at least $\tfrac{k}{3}$ colors, then $|\varphi(x)|\le \tfrac{k}{3}+\tfrac{2k}{9}$.

Before proceeding with the proofs, let us mention another consequence of Theorem~\ref{x-tri-free}.  Consider a triangle-free planar graph $G$
with $V(G)=\{v_1,v_2,\ldots,v_n\}$. For each vertex $v_i\in V(G)$, the graph $G$ has a set coloring $\varphi_i$ by subsets of
$\{6i-5,6i-4,6i-3,6i-2,6i-1,6i\} $ such that $|\varphi(v_j)|\geq 2$ for $1\leq j\leq n$ and $|\varphi(v_i)|=3$.
Let us set $\varphi(v)=\bigcup_{i=1}^{n}\varphi_i(v)$ for each $v\in V(G)$.  Then $\varphi$ is a set coloring of $G$ by $(2n+1)$-element
subsets of $\{1,\ldots,6n\}$.  Hence, we have the following corollary, which improves upon Theorem~\ref{(9n:3n+1)}.

\begin{corollary}
\label{frac-tri-free}
Every triangle-free planar graph on $n$ vertices is $(6n:2n+1)$-colorable, and thus its fractional chromatic number is
at most $3-\frac{3}{2n+1}$.
\end{corollary}

\section{Set coloring of triangle-free planar graphs}

In this section, we give a proof of Theorem~\ref{x-tri-free}.
Let $G$ be a graph and let $X$ be a set of vertices of $G$.
An {\em $X$-enhanced coloring} of $G$ is a set coloring $\varphi$ of $G$ by subsets of $\{1,\ldots,6\}$
such that $|\varphi(v)|\geq 2$ for all $v\in V(G)$ and $|\varphi(x)|=3$ for $x\in X$.
We are going to prove a mild strengthening of Theorem~\ref{x-tri-free} where the outer face is precolored.

\begin{theorem}
\label{X-tri-free}
Let $G$ be a triangle-free plane graph whose outer face is bounded by a cycle $C$ of length at most 5,
and let $X$ be a subset of $V(C)$ of size at most one. Then any $X$-enhanced coloring of $C$ can be extended
to an $X$-enhanced coloring of $G$.
\end{theorem}

Theorem~\ref{x-tri-free} follows from Theorem~\ref{X-tri-free} by redrawing the graph so that $x$
is incident with the outer face, adding three new vertices $v_1$, $v_2$, and $v_3$ and the edges
of the $4$-cycle $C=xv_1v_2v_3$ bounding the outer face of the resulting graph, letting $X=\{x\}$
and choosing an $X$-enhanced coloring of $C$ arbitrarily.

A (hypothetical) counterexample to Theorem~\ref{X-tri-free} is a triple
$(G,X,\varphi)$, where $G$ is a triangle-free plane graph whose outer face is
bounded by a cycle $C$ of length at most 5, $X$ is a subset of $V(C)$ with
$|X|\leq 1$, and $\varphi$ is an $X$-enhanced coloring of $C$ such that
$\varphi$ does not extend to an $X$-enhanced coloring of $G$.
The counterexample $(G,X,\varphi)$ is \emph{minimal} if there is no
counterexample $(G',X',\varphi')$ such that either $|V(G')|<|V(G)|$,
or $|V(G')|=|V(G)|$ and $|E(G')|>|E(G)|$; i.e., $G$ has the minimum number
of vertices among all counterexamples, and the maximum number of edges
among all counterexamples with the minimum number of vertices.

\subsection{Properties of a minimal counterexample}

Let us start with some observations on vertex degrees and face lengths in a minimal counterexample.

\begin{lemma}\label{basic}
If $(G, X, \varphi)$ is a minimal counterexample, then $G$ is 2-connected,
all vertices of degree two are incident with the outer face or adjacent to a vertex in $X$, and
every $(\leq\!5)$-cycle in $G$ bounds a face.
\end{lemma}
\begin{proof}
Let $v$ be a vertex of $G$ of degree at most two, not contained in the cycle $C$ bounding
the outer face of $G$.  Since $(G,X,\varphi)$ is a minimal counterexample, the coloring
$\varphi$ extends to an $X$-enhanced coloring $\psi$ of $G-v$.  Since $\psi$ does not
extend to an $X$-enhanced coloring of $G-v$, we conclude that $|\bigcup_{uv\in E(G)}\psi(u)|\ge 5$,
and thus $\deg(v)=2$ and $v$ is adjacent to a vertex in $X$.

Suppose now that $G$ is not 2-connected, and thus there exist proper induced subgraphs $G_1$ and $G_2$ of $G$
intersecting in at most one vertex such that $G=G_1\cup G_2$ and $C \subseteq G_1$.
Let $f$ be a face of $G$ incident with both a vertex of $V(G_1)\setminus V(G_2$
and a vertex of $V(G_2)\setminus V(G_1)$. Since $G$ is triangle-free and
has minimum degree at least two, observe that for $i\in \{1,2\}$, there exists a vertex
$v_i\in  V(G_i)$ incident with $f$ such that if $G_1$ and $G_2$ intersect,
then the distance between $v_i$ and the vertex in $G_1\cap G_2$ is at least two.
Then $G+v_1v_2$ is triangle-free and has more edges than $G$, and thus by the minimality
of $(G,X,\varphi)$, there exists an $X$-enhanced coloring of $G+v_1v_2$ extending $\varphi$.
This also gives an $X$-enhanced coloring of $G$, which is a contradiction.  Hence, $G$ is $2$-connected.

Suppose that a $(\leq\!5)$-cycle $K$ of $G$ does not bound a face. Since $G$ is
triangle-free, the cycle $K$ is induced. Let $G_1$ be the subgraph of $G$ drawn
outside (and including) $K$, and let $G_2$ be the subgraph of $G$ drawn inside
(and including) $K$. We have $V(G_1)<V(G)$, and thus there exists an $X$-enhanced coloring $\varphi_1$ of $G_1$ extending $\varphi$.
Furthermore, since $|V(G_2)|<|V(G)|$, there exists an $X$-enhanced coloring $\varphi_2$ of $G_2$ that matches $\varphi_1$ on $K$.
The union of $\varphi_1$ and  $\varphi_2$ is an $X$-enhanced coloring of $G$ extending $\varphi$, which is a contradiction. Hence,
every $(\leq\!5)$-cycle of $G$ bounds a face.
\end{proof}

\begin{lemma}
\label{4-face}
If $(G, X, \varphi)$ is a minimal counterexample with the outer face bounded
by a cycle $C$, then $G$ contains no 4-cycle other than $C$.
\end{lemma}

\begin{proof}
Suppose that $G$ contains a 4-cycle $K=v_1v_2v_3v_4$ distinct from $C$.  By Lemma~\ref{basic}, $K$ bounds
a face. Since $K\neq C$, we can assume that $v_3\notin V(C)$. Let $G'$ be the graph
obtained from $G$ by identifying $v_1$ with $v_3$. Note that each $X$-enhanced coloring of $G'$
corresponds to an $X$-enhanced coloring of $G$, and thus $\varphi$ does not extend to an $X$-enhanced coloring
of $G'$. Since $|V(G')|<|V(G)|$, we conclude by the minimality of $(G, X, \varphi)$ that $G'$ contains a triangle.
Hence, $G$ contains a 5-cycle $Q=v_1v_2v_3uw$. By Lemma~\ref{basic}, the $5$-cycles $Q$ and $Q'=v_1v_4v_3uw$ bound faces.
We conclude that $G$ has only three faces, bounded by the cycles $K$, $Q$, and $Q'$.  However, $v_3\in V(K\cap Q\cap Q')$,
but we chose $v_3$ not to be incident with the outer face of $G$, which is a contradiction.
\end{proof}

A \emph{$k$-face} is a face of length exactly $k$, and a \emph{$k$-vertex} is a vertex of degree exactly $k$.
A \emph{$k^+$-face} is a face of length at least $k$, and a \emph{$k^+$-vertex} is a vertex of degree at least $k$.

\begin{lemma}\label{6-face}
If $(G, X, \varphi)$ is a minimal counterexample, then $G$ contains no $6^+$-faces.
\end{lemma}
\begin{proof}
Suppose for a contradiction that $G$ contains a $6^+$-face bounded by a cycle $K=v_1\ldots v_k$, where $k\geq 6$.
Since the outer face of $G$ is bounded by a cycle $C$ of length at most five,
we can choose the labeling of vertices of $K$ so that $v_1\notin V(C)$.
By Lemma~\ref{4-face}, $v_1v_4\notin E(G)$.
Let $G'=G+v_1v_4$. If $G'$ contained a triangle, then $G$ would contain a 5-cycle $Q=v_1v_2v_3v_4u$,
which would bound a face by Lemma~\ref{basic}. Hence, the path $v_1v_2v_3v_4$
would be contained in boundaries of two distinct faces of $G$, and thus $v_2$ and $v_3$ would have degree two.
Since $v_1\not\in V(C)$, we would also have $v_2,v_3\not\in V(C)$, and thus $v_2$ would be a vertex of degree two
not contained in $C$ and not adjacent to $X$, contradicting Lemma~\ref{basic}.  Hence, $G'$ is triangle-free,
and $(G',X,\varphi)$ is a counterexample contradicting the minimality of $(G,X,\varphi)$.
\end{proof}

By Lemmas~\ref{4-face} and~\ref{6-face}, we have the following corollary.

\begin{corollary}
\label{5-face}
If  $(G, X, \varphi)$ is a minimal counterexample, then every face other than the outer one is a 5-face.
\end{corollary}

Next, we prove two claims restricting the $5$-faces.

\begin{lemma}
\label{tie-5-face}
Let $(G, X, \varphi)$ be a minimal counterexample with the outer face bounded
by a cycle $C$. Let $K=v_1v_2v_3v_4v_5$ be a cycle bounding a 5-face in
$G$ such that $v_1$, $v_2$, $v_3$ and $v_4$ have degree three and do not belong to $V(C)$.
For $i\in\{1,\ldots,4\}$, let $u_i$ denote the neighbor of $v_i$ not belonging to $V(K)$.
Then either $\{u_1,\ldots,u_4,v_5\}\cap X\neq \emptyset$ or $|\{u_1,\ldots, u_4,v_5\}\cap V(C)|\ge 2$.
\end{lemma}
\begin{proof}
Suppose for a contradiction that $u_1,\ldots,u_4,v_5\not\in X$ and
at most one of the vertices $u_1$, \ldots, $u_4$, and $v_5$
belongs to $V(C)$.  If $u_i=u_j$ for distinct $i,j\in\{1,\ldots,4\}$, then $v_i$ and $v_j$
are contained in a triangle or a $4$-cycle.  The former is not possible, since $G$ is triangle-free.
In the latter case, Lemma~\ref{4-face} implies this $4$-cycle is $C$,
contradicting the assumption that $v_i\not\in V(C)$.  Therefore, the vertices $u_1$, \ldots, $u_4$ are
pairwise distinct.

Suppose that $G$ contains an edge $u_iu_j$ for distinct $i,j\in\{1,\ldots,4\}$.  Analogously to the previous paragraph,
this is not possible when $|i-j|=1$.  If $|i-j|=2$, then let $k=(i+j)/2$, otherwise (when $\{i,j\}=\{1,4\}$), let $k=5$.
Then $G$ contains a $5$-cycle $u_iv_iv_kv_ju_j$, and by Lemma~\ref{basic} this $5$-cycle bounds a face, implying that $v_k$
has degree two.  Since $v_i,v_j\not\in X$ and $|V(K)\cap V(C)|\le 1$, this contradicts Lemma~\ref{basic}.
Therefore, the vertices $u_1$, \ldots, $u_4$ are pairwise non-adjacent.

Next, we show that for $i\in\{1,2,3\}$, the graph obtained from $G-\{v_i,v_{i+1}\}$
by identifying $u_i$ and $u_{i+1}$ is triangle-free. Otherwise, $G$ contains a 6-cycle
$Q=v_iv_{i+1}u_{i+1}w_{i+1}w_iu_i$.  By Corollary~\ref{5-face}, since $\deg(v_i)=\deg(v_{i+1})=3$,
$G$ has a $5$-face bounded by a $5$-cycle $u_iv_iv_{i+1}u_{i+1}y_i$, and by Lemma~\ref{basic},
the $5$-cycle $u_iw_iw_{i+1}u_{i+1}y_i$ also bounds a face.  Consequently, $y_i$ has degree two,
and by Lemma~\ref{basic}, we conclude that either $y_i$ has a neighbor in $X$ or $y_i\in V(C)$.
However, then either $\{u_i,u_{i+1}\}\cap X\neq\emptyset$ or $u_i,u_{i+1}\in V(C)$, which is a contradiction.

Let $G'$ be the graph obtained from $G-\{v_1,v_2,v_3,v_4\}$ by adding the edge $u_1u_4$ and by
identifying $u_2$ with $u_3$.
If $G'$ is triangle-free, then by the minimality of $(G,X,\varphi)$,
there exists an $X$-enhanced coloring $\psi$ of $G'$ extending $\varphi$. Note
that $\psi(u_1)\cap \psi(u_4)=\emptyset$ and we can assume that $|\psi(u_1)|=\ldots=|\psi(u_4)|=|\psi(v_5)|=2$. Hence, we can let $\psi(v_1)$ be a 2-element subset of
$\{1,\ldots,6\}\setminus (\psi(u_1)\cup \psi(v_5))$ and $\psi(v_4)$ a 2-element subset of
$\{1,\ldots,6\}\setminus (\psi(u_4)\cup \psi(v_5))$
such that $\psi(v_1)\cap \psi(v_4)=\emptyset$. Since $\psi(u_2)=\psi(u_3)$,
$\psi$ can be extended to $v_2$ and $v_3$. This gives an $X$-enhanced coloring of $G$
extending $\varphi$, which is a contradiction.
So $G'$ has a triangle, necessarily containing the edge $u_1u_4$.  Since $u_1u_2,u_3u_4\notin E(G)$, the vertex
obtained by identifying $u_2$ with $u_3$ is not contained in the triangle.
Hence, $u_1$ and $u_4$ have a common neighbor $w$ in $G$.

Let $G''$ be the graph obtained from $G-\{v_1,v_2,v_3,v_4\}$ by identifying
$u_1$ with $u_2$, and $u_3$ with $v_5$. If $G''$ is triangle-free,
then there exists an $X$-enhanced coloring $\psi$ of $G'$ extending
$\varphi$ by the minimality of $(G,X,\varphi)$.
We can assume $|\psi(u_1)|=\ldots=|\psi(u_4)|=|\psi(v_5)|=2$,
and thus $\psi$ can be extended to $v_4$ and $v_3$.  Note that $\psi(v_5)=\psi(u_3)$, and thus $\psi(v_5)\cap
\psi(v_3)=\emptyset$, enabling us to extend $\psi$ to $v_1$ and $v_2$.  This gives an
$X$-enhanced coloring of $G$ extending $\varphi$, which is a contradiction.

Therefore, $G''$ has a triangle, necessarily containing the vertex obtained by the identification
of $u_3$ with $v_5$.  Since $u_1u_3,u_1v_5\notin E(G)$, we conclude that the triangle does not contain
the vertex obtained by the identification of $u_1$ with $u_2$, and thus
$G-\{v_1,v_2,v_3,v_4\}$ contains a path
$u_3xyv_5$.  Note that $u_1,u_4\not\in\{x,y\}$, since $G$ is triangle-free and $u_1u_3,u_3u_4\not\in E(G)$.
Since $v_1,v_4\not\in V(C)$, Lemma~\ref{4-face} implies $u_1y,u_4y\not\in E(G)$.
Since $w$ is a common neighbor of $u_1$ and $u_4$, by planarity we conclude that $w=x$, and thus $w$ is adjacent to $u_3$.
By a symmetric argument applied to the graph obtained from $G-\{v_1,v_2,v_3,v_4\}$ by identifying
$u_3$ with $u_4$, and $u_2$ with $v_5$, we conclude that $w$ is also adjacent to $u_2$.
However, then Lemma~\ref{basic} implies that $G$ has exactly $6$ faces, bounded by $K$,
$wyv_5v_1u_1$, $wyv_5v_4u_4$, and $wu_iv_iv_{i+1}u_{i+1}$ for $i\in \{1,2,3\}$.
One of these $5$-cycles is $C$, implying that $|\{u_1,\ldots, u_4,v_5\}\cap V(C)|\ge 2$,
which is a contradiction.
\end{proof}

\begin{corollary}
\label{x-tie-5-face}
Let $(G, X, \varphi)$ be a minimal counterexample with the outer face bounded by a cycle $C$ and with $X=\{x\}$.
Let $K=v_1v_2v_3v_4v_5$ be a cycle in $G$ vertex-disjoint from $C$ such that $\deg(v_5)=3$ and $xv_5\in E(G)$.
Then at least one of vertices $v_1$, \ldots, $v_4$ has degree at least four.
\end{corollary}
\begin{proof}
By Lemma~\ref{basic}, $K$ bounds a face.  For $i\in \{1,\ldots, 4\}$, Lemma~\ref{4-face} and the assumption
that $G$ is triangle-free implies $xv_i\not\in E(G)$, and thus $\deg(v_i)\ge 3$ by Lemma~\ref{basic}.
Suppose for a contradiction that $\deg(v_i)=3$ for $i\in\{1,\ldots,4\}$.
Let $u_i$ denote the neighbor of $v_i$ not in $V(K)$.
As in the proof of Lemma~\ref{tie-5-face}, we argue that the vertices $u_1$, \ldots, $u_4$
are pairwise disjoint and non-adjacent.

By Lemma~\ref{tie-5-face}, two of the vertices $u_1$, \ldots, $u_4$ belong to $V(C)$. Consequently, at least one of them
is adjacent to $x$. By symmetry, we can assume that there exists $i\in\{1,2\}$
such that $u_i\in V(C)$ and $u_i$ is adjacent to $x$.  If $i=1$, then Lemma~\ref{4-face}
applied to the $4$-cycle $u_1v_1v_5x$ implies $v_1\in V(C)$, which is a contradiction.  If $i=2$,
then the $5$-cycle $u_2v_2v_1v_5x$ bounds a face by Lemma~\ref{basic}, and thus $\deg(v_1)=2$, which is again a contradiction.
\end{proof}

\subsection{Reducible configurations}

Let us now derive further properties of special configurations in a minimal counterexample.
We will often need the following observation.
\begin{observation}\label{int123}
If $\psi$ is an $\{x\}$-enhanced coloring of a path $xuv$, then
$\psi(x)\cap \psi(v)\neq\emptyset$.  Conversely, any precoloring $\psi'$ of $x$ and $v$
such that $|\psi'(x)|=3$, $|\psi'(v)|=2$, and $\psi'(x)\cap \psi'(v)\neq\emptyset$ extends
to an $\{x\}$-enhanced coloring of the path.
\end{observation}

Next, we restrict degrees of vertices near to $X$.

\begin{lemma}
\label{233}
Let $(G, X, \varphi)$ be a minimal counterexample with the outer face bounded by a cycle $C$ and with $X=\{x\}$.
Suppose a cycle $xv_1u_1u_2v_2$ bounds a 5-face in $G$.
If $\deg(v_1)=2$, $\deg(u_1)=3$, and $v_1,u_1,u_2\not\in V(C)$, then $\deg(v_2)=2$ and $v_2\not\in V(C)$.
\end{lemma}
\begin{proof}
Let $G'$ be the graph obtained from $G-\{v_1,u_1\}$ by identifying $x$ and $u_2$.
Suppose first that $G'$ has an $X$-enhanced coloring $\psi'$ extending $\varphi$.
We may assume without lose of generality that $\psi'(x)=\{1,2,3\}$.
Let $u_0$ denote the neighbor of $u_1$ distinct from $v_1$ and $u_2$.
By Corollary~\ref{5-face}, $u_0$ has a common neighbor with $x$, and by Observation~\ref{int123},
we can without lose of generality assume $1\in\psi'(u_0)$.
Furthermore, by symmetry between the colors $2$ and $3$, we can assume $3\not\in \psi'(u_0)$.
Let $\psi(v)=\psi'(v)$ for $v\in V(G)\setminus \{v_1,u_1,u_2\}$.
Let $\psi(u_2)=\{1,2\}$ (this is a subset of $\psi'(u_2)=\psi'(x)=\{1,2,3\}$),
let $\psi(u_1)$ be a $2$-element subset of $\{3,\ldots,6\}\setminus \psi'(u_0)$ containing $3$,
and extend $\psi$ to $v_1$ by Observation~\ref{int123}.
Then $\psi$ is an $X$-enhanced coloring of $G$ extending $\varphi$, which is a contradiction.

Consequently, $G'$ does not have an $X$-enhanced coloring extending $\varphi$,
and by the minimality of $(G,X,\varphi)$, we conclude $G'$ contains a triangle.
Hence, $G$ contains a 5-cycle $xv_2u_2w_2w_1$ disjoint from $\{u_1,v_1\}$, and by Lemma~\ref{basic}, this $5$-cycle
bounds a face.  Hence, $v_2$ has degree two.  Since $u_2\not\in V(C)$, we conclude $v_2\not\in V(C)$.
\end{proof}

\begin{lemma}
\label{232325}
Let $(G, X, \varphi)$ be a minimal counterexample with the outer face bounded by a cycle $C$ and with $X=\{x\}$.  Let $xv_1u_1u_2v_2$ and $xv_2u_2u_3v_3$ be distinct cycles bounding 5-faces in
$G$ such that $u_1,u_2,u_3\not\in V(C)$.  If $\deg(u_1)=\deg(u_2)=3$, then $\deg(u_3)\geq 5$.
\end{lemma}
\begin{proof}
Note that $\deg(v_2)=2$ and $v_2\not\in V(C)$.  By Lemma~\ref{233}, we have $\deg(v_1)=\deg(v_3)=2$ and $v_1,v_3\not\in V(C)$.
By Lemma~\ref{basic}, $\deg(u_3)\geq 3$. Suppose for a contradiction that
$\deg(u_3)\leq 4$.  By Corollary~\ref{5-face}, there exist paths $u_1u_0v_0x$ and $u_3u_4v_4x$
in $G$ with $u_0\neq u_2\neq u_4$.

First consider the case $\deg(u_3)=3$.
By the minimality of $(G, X, \varphi)$, the graph $G'=G-\{v_2,u_2\}$ has an $X$-enhanced coloring $\psi'$ extending $\varphi$.
We may assume without lose of generality that $\psi'(x)=\{1,2,3\}$. If $|(\psi'(u_1)\cup \psi'(u_3))\cap \{1,2,3\}|\leq 2$,
then $\psi'$ can be extended to $u_2$ and $v_2$ by Observation~\ref{int123}. This gives an $X$-enhanced coloring of $G$ extending $\varphi$, which is a contradiction.
Therefore, we can assume $\psi'(u_1)=\{1,2\}$ and $3\in \psi'(u_3)$.
By Observation~\ref{int123}, we can assume $\psi'(u_0)=\{3,4\}$.
Furthermore, by symmetry between colors $1$ and $2$, and between colors $5$ and $6$, we can assume $1,6\notin \psi'(u_3)$.
Let $\psi(v)=\psi'(v)$ for $v\in V(G)\setminus \{v_1,u_1,v_2,u_2\}$.  Set $\psi(u_1)=\{2,5\}$, $\psi(v_1)=\{4,6\}$,
$\psi(u_2)=\{1,6\}$ and $\psi(v_2)=\{4,5\}$.
Then $\psi$ is an $X$-enhanced coloring of $G$ extending $\varphi$, which is a contradiction.

Now we assume $\deg(u_3)=4$.  Let $w$ be the neighbor of $u_3$ distinct from $u_2$, $v_3$, and $u_4$.
By Lemma~\ref{4-face}, $xw\notin E(G)$, and in particular $w\neq v_0$.  By Corollary~\ref{5-face}, we have $wu_0\in E(G)$,
and in particular $\deg(u_0)\ge 3$.
Let $G'$ be the graph obtained from $G-\{v_2,u_2\}$ by identifying $u_1$ and $w$.
By Lemma~\ref{basic}, since $\deg(u_0)\ge 3$, $G-u_2$ does not contain a path of length three between $u_1$ and $w$,
and thus $G'$ is triangle-free.
By the minimality of $(G, X, \varphi)$, there exists an $X$-enhanced coloring $\psi'$ of $G'$ extending $\varphi$.
We may assume without lose of generality that $\psi'(x)=\{1,2,3\}$.
If $|(\psi'(u_1)\cup \psi'(u_3))\cap \{1,2,3\}|\leq 2$, then $\psi'$ extends to $u_2$ and $v_2$ by Observation~\ref{int123}.
This gives an $X$-enhanced coloring of $G$ extending $\varphi$, which is a contradiction.
Therefore $\{1,2,3\}\subseteq \psi'(u_1)\cup \psi'(u_3)$.

If $|\psi'(u_3)\cap \{1,2,3\}|=2$, we can assume $\psi'(u_3)=\{1,2\}$ and $3\in \psi'(u_1)=\psi'(w)$.
By Observation~\ref{int123}, we can also assume $\psi'(u_4)=\{3,4\}$.
By symmetry between the colors $1$ and $2$, and between the colors $5$ and $6$, we can assume $1,6\not\in \psi'(u_1)=\psi'(w)$.
Let $\psi(v)=\psi'(v)$ for $v\in V(G)\setminus \{v_2,u_2,v_3,u_3\}$,
$\psi(u_3)=\{2,6\}$, $\psi(v_3)=\{4,5\}$, $\psi(u_2)=\{1,\alpha\}$ for a color $\alpha\in\{4,5\}\setminus\psi'(u_1)$,
and $\psi(v_2)=\{9-\alpha,6\}$. This gives an $X$-enhanced coloring of $G$ extending $\varphi$,
which is a contradiction.

Hence $|\psi'(u_3)\cap \{1,2,3\}|=1$, and we can assume $\psi'(u_3)=\{3,4\}$ and $\psi'(u_1)=\psi'(w)=\{1,2\}$.
By Observation~\ref{int123}, we have $3\in \psi'(u_0)$, and thus $|\{5,6\}\cap \psi'(u_0)|\le 1$ and by
symmetry between the colors $5$ and $6$, we can assume that $6\not\in \psi'(u_0)$.
Let $\psi(v)=\psi'(v)$ for $v\in V(G)\setminus \{v_2,u_2,v_1,u_1\}$,
$\psi(u_1)=\{2,6\}$, $\psi(v_1)=\{4,5\}$, $\psi(u_2)=\{1,5\}$, and $\psi(v_2)=\{4,6\}$.
This gives an $X$-enhanced coloring of $G$ extending $\varphi$, which is a contradiction.
\end{proof}

\begin{lemma}
\label{232424}
Let $(G, X, \varphi)$ be a minimal counterexample with the outer face bounded by a cycle $C$ and with $X=\{x\}$.  Let $xv_1u_1u_2v_2$ and $xv_2u_2u_3v_3$ be cycles bounding 5-faces in $G$,
such that $\deg(v_1)=\deg(v_2)=\deg(v_3)=2$ and $\deg(u_1)=\deg(u_3)=3$.
If $u_1,u_2,u_3\not\in V(C)$, then $\deg(u_2)\ge 5$.
\end{lemma}
\begin{proof}
By Corollary~\ref{5-face}, there exist paths $u_1u_0v_0x$ and $u_3u_4v_4x$ in $G$ with $u_0\neq u_2\neq u_4$.
By Lemma~\ref{basic}, $\deg(u_2)\geq 3$, and by Lemma~\ref{232325}, $\deg(u_2)\ge 4$.
Suppose for a contradiction that $\deg(u_2)=4$, and let $w$ be the neighbor of $u_2$ distinct from
$u_1$, $v_2$, and $u_3$. By Lemma~\ref{basic}, $xw\notin E(G)$, and $w$ and $x$ have no common neighbor.
Let $G'=G-\{v_2,u_2\}+xw$.  Then $G'$ is triangle-free.
By the minimality of $(G, X, \varphi)$, there exists an $X$-enhanced coloring $\psi'$ of $G'$ extending $\varphi$.
We may assume without lose of generality that $\psi'(x)=\{1,2,3\}$ and $\psi'(w)=\{4,5\}$.

Let $\psi$ be the restriction of $\psi'$ to $G-\{v_1,u_1,v_2,u_2,v_3,u_3\}$.
By Observation~\ref{int123}, we can assume $1\in \psi(u_0)$, and thus by symmetry between the colors $2$ and $3$,
and between the colors $4$ and $5$,
we can assume that $\{2,5\}\cap \psi(u_0)=\emptyset$.  Set $\psi(u_1)=\{2,5\}$ and $\psi(v_1)=\{4,6\}$.
By a symmetric argument, there exist $\alpha\in \{1,2,3\}$ and $\beta\in \{4,5\}$ such that
$\{\alpha,\beta\}\cap \psi(u_4)=\emptyset$.  Set $\psi(u_3)=\{\alpha,\beta\}$ and $\psi(v_3)=\{9-\beta,6\}$.
Let $\gamma$ be a color in $\{1,3\}\setminus\{\alpha\}$ and set $\psi(u_2)=\{\gamma,6\}$ and $\psi(v_2)=\{4,5\}$.
This gives an $X$-enhanced coloring of $G$ extending $\varphi$, which is a contradiction.
\end{proof}

\subsection{More reducible configurations}

Before we proceed with our analysis of configurations in a minimal counterexample,
let us establish an auxiliary result on colorings of the graph depicted in Figure~\ref{fig-H}.

\begin{figure}[!htb]
\centering
{\includegraphics[height=0.33\textwidth]{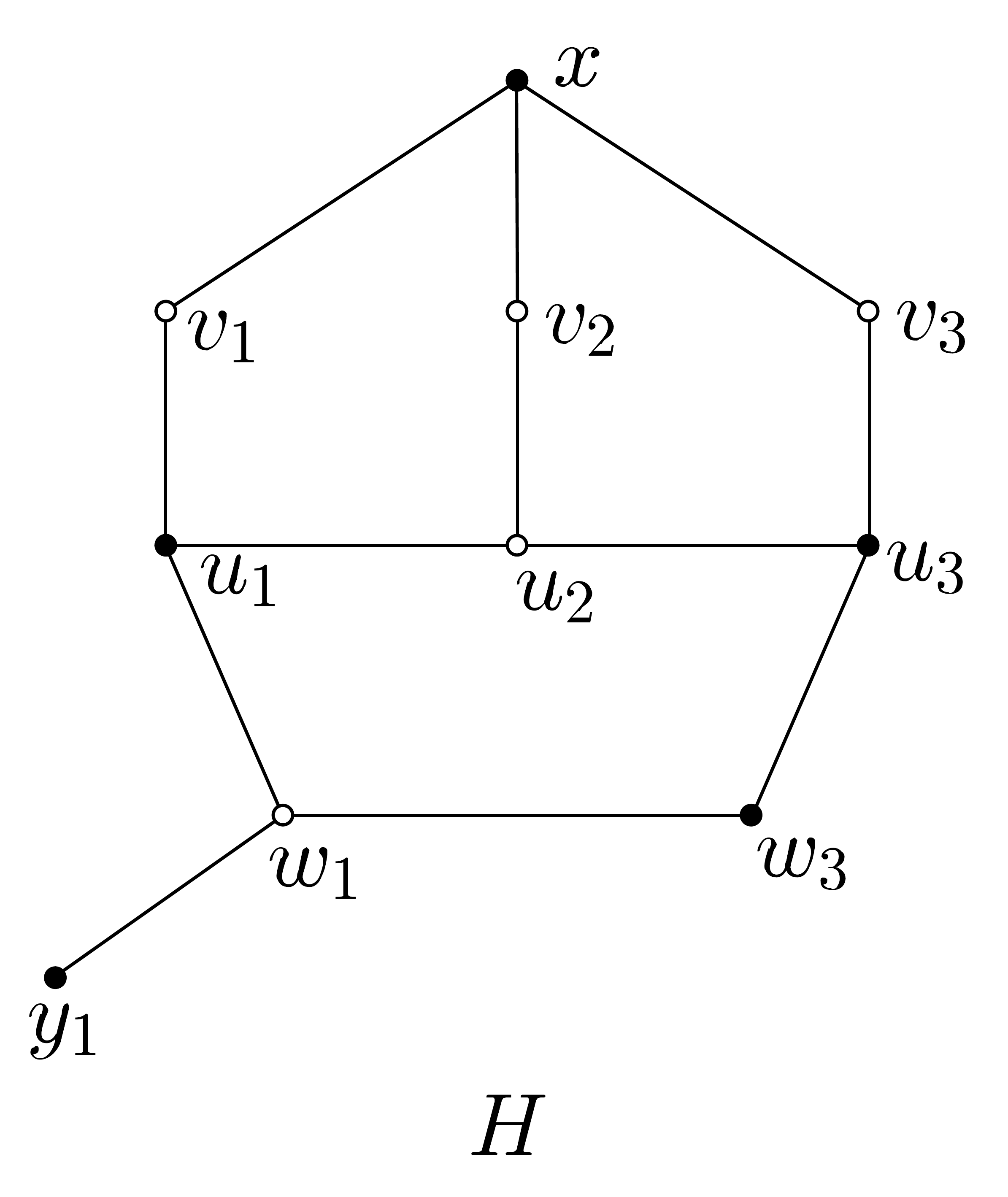}}

\caption{The graph $H$ from the statement of Lemma~\ref{spe-set}.}\label{fig-H}
\end{figure}

\begin{lemma}
\label{spe-set}
Let $H$ be the graph shown in Figure~\ref{fig-H} and let $L$ be an assignment of subsets of $\{1,\ldots,6\}$
to vertices of $H$ satisfying the following conditions: $L(x)=\{1,2,3\}$, $|L(u_1)|=3$ and $L(u_1)\cap\{1,2,3\}=\{3\}$,
$L(u_3)=\{1,2,5,6\}$, $|L(w_3)|=4$, $|L(y_1)|=2$ and $L(y_1)\subseteq \{3,4,5,6\}$, and
$L(v_1)=L(v_2)=L(v_3)=L(u_2)=L(w_1)=\{1,\ldots,6\}$.  There exists a $2$-element set $S\subseteq L(u_1)$
such that $3\in S$ and $S\cap L(y_1)\neq\emptyset$, and for any such set $S$,
the graph $H$ has an $\{x\}$-enhanced coloring $\varphi$ such that $\varphi(v)\subseteq L(v)$ for all $v\in V(H)$ and $\varphi(u_1)=S$.
\end{lemma}
\begin{proof}
Since $|L(u_1)|=3$, $|L(y_1)|=2$ and $L(u_1),L(y_1)\subseteq \{3,4,5,6\}$,
we have $L(u_1)\cap L(y_1)\neq \emptyset$.  Hence, there exists a $2$-element set $S\subseteq L(u_1)$
such that $3\in S$ and $S\cap L(y_1)\neq\emptyset$.

Consider any such set $S$, and let $\varphi(u_1)=S$, $\varphi(x)=\{1,2,3\}$ and $\varphi(y_1)=L(y_1)$.
Then $\varphi(u_1)\cup \varphi(y_1)\subseteq \{3,4,5,6\}$ and $|\varphi(u_1)\cup \varphi(y_1)|\leq 3$.
Let $\varphi(v_1)$ be a $2$-element subset of $\{4,5,6\}\setminus S$.

If there exists a color $\alpha\in L(w_3)\setminus L(u_3)$, then let
$\varphi(w_1)$ be a 2-element subset of $L(w_1)\setminus (\varphi(u_1)\cup
\varphi(y_1)\cup \{\alpha\})$. Let  $\varphi(w_3)$  be a 2-element subset of
$L(w_3)\setminus \varphi(w_1)$ such that $\alpha\in L(w_3)$. Then
$|L(u_3)\setminus \varphi(w_3)|\geq 3$. Thus we can choose a 2-element subset
$\varphi(u_3)$ of $L(u_3)\setminus \varphi(w_3)$ such that $|\varphi(u_3)\cap \{1,2\}|=1$;
by symmetry, we can assume that $\varphi(u_3)=\{1,5\}$.  Since $3\in\psi(u_1)$, we have $|\{4,6\}\cap \psi(u_1)|\le 1$,
and thus we can assume that say $4\not\in \psi(u_1)$.
Set $\varphi(u_2)=\{2,4\}$, $\varphi(v_2)=\{5,6\}$, and $\varphi(v_3)=\{4,6\}$.
This gives a set coloring of $H$ as required.

Hence, we can assume $L(w_3)=L(u_3)=\{1,2,5,6\}$.
Let $\varphi(w_1)$ be a 2-element subset of $L(w_1)\setminus (\varphi(u_1)\cup \varphi(y_1))$
such that $1\in \varphi(w_1)$ and $2\notin \varphi(w_1)$. Choose
$\varphi(w_3)$ as a 2-element subset of $L(w_3)\setminus \varphi(w_1)$ containing the color $2$;
by symmetry, we can assume $\varphi(w_3)=\{2,5\}$.  Let $\varphi(u_3)=\{1,6\}$ and $\varphi(v_3)=\{4,5\}$.
Let $\varphi(u_2)$ be a $2$-element subset of $\{2,3,4,5\}\setminus\varphi(u_1)$ containing the color $2$,
and let $\varphi(v_2)$ be a $2$-element subset of $\{4,5,6\}\setminus\varphi(u_2)$.
This again gives a set coloring of $H$ as required.
\end{proof}

\begin{lemma}
\label{242324}
Let $(G, X, \varphi)$ be a minimal counterexample with the outer face bounded by a cycle $C$ and with $X=\{x\}$.  Let $xv_1u_1u_2v_2$, $xv_2u_2u_3v_3$, and $u_1u_2u_3w_3w_1$ be cycles bounding distinct 5-faces in $G$.
If $u_1,u_2,u_3\notin V(C)$ and $\deg(u_1)=\deg(u_3)=4$, then $w_1,w_4\not\in V(C)$ and $\max(\deg(w_1),\deg(w_3))\geq 4$.
\end{lemma}
\begin{proof}
Note that $\deg(u_2)=3$, $\deg(v_2)=2$, and $v_2\not\in V(C)$, and thus $\deg(v_1)=\deg(v_3)=2$ and $v_1,v_3\notin V(C)$
by Lemma~\ref{233}.
By Corollary~\ref{5-face}, there exist paths $u_1u_0v_0x$ and $u_3u_4v_4x$ in $G$ with $u_0\neq u_2\neq u_4$.
If $w_i\in V(C)$ for some $i\in \{1,3\}$, then by Lemma~\ref{basic}, the cycle formed by the path $xv_iu_iw_i$
together with a path of length at most two between $x$ and $w_i$ in $C$ would bound a face, contradicting the assumption $\deg(u_i)=4$.
Hence, $w_1,w_3\not\in V(C)$.  Furthermore, $w_1x,w_2x\not\in E(G)$ by Lemma~\ref{4-face}.  Hence, Lemma~\ref{basic} implies $\deg(w_1),\deg(w_3)\geq 3$.

Suppose for a contradiction that $\deg(w_1)=\deg(w_3)=3$.  For $i\in \{1,3\}$, let $y_i$ be the neighbor of $w_i$ distinct
from $u_i$ and $w_{4-i}$.
Let $G'$ be the graph obtained from $G-\{v_2,u_2\}$ by identifying $u_3$ and
$w_1$.  Since $\deg(w_3)=3$, Lemma~\ref{basic} implies that $u_3$ and $w_1$ are not
joined by a path of length three in $G-u_2$, and thus $G'$ is triangle-free.
By the minimality of $(G,X,\varphi)$, there exists an $X$-enhanced
coloring $\psi'$ of $G'$ extending $\varphi$. We may assume without lose of
generality that $\psi'(x)=\{1,2,3\}$. If $|(\psi'(u_1)\cup \psi'(u_3))\cap \{1,2,3\}|\leq 2$, then $\psi'$
can be extended to $u_2$ and $v_2$ by Observation~\ref{int123}. This gives an $X$-enhanced coloring of $G$ extending $\varphi$, which is a contradiction.
Therefore, $\{1,2,3\}\subseteq \psi'(u_1)\cup \psi'(u_3)$.

Suppose first $|\psi'(u_1)\cap \{1,2,3\}|=2$, and thus we can assume $\psi'(u_1)=\{1,2\}$
and $3\in \psi'(u_3)=\psi'(w_1)$.  By Observation~\ref{int123},
we can assume $\psi'(u_0)=\{3,4\}$.  By symmetry between the colors $5$ and $6$,
we can assume $6\not\in \psi'(u_3)=\psi'(w_1)$.
Let $\psi(v)=\psi'(v)$ for $v\in V(G)\setminus \{u_1,v_1, u_2,v_2\}$,
$\psi(u_1)=\{2,6\}$, $\psi(v_1)=\{4,5\}$, $\psi(u_2)=\{1,\alpha\}$ for a color $\alpha\in\{4,5\}\setminus \psi'(u_3)$,
and $\psi(v_2)=\{9-\alpha,6\}$.
This gives an $X$-enhanced coloring of $G$
extending $\varphi$, which is a contradiction.

Therefore, $|\psi'(u_1)\cap \{1,2,3\}|=1$, and we can assume $\psi(u_1)\cap \{1,2,3\}=\{3\}$ and $\psi(u_3)=\psi(w_1)=\{1,2\}$.
By Observation~\ref{int123}, we can assume $\psi'(u_4)=\{3,4\}$ and $|\psi'(u_0)\cap \{4,5,6\}|\le 1$.
Let $L(x)=\{1,2,3\}$, $L(v_1)=L(v_2)=L(v_3)=L(u_2)=L(w_1)=\{1,\ldots,6\}$, $L(u_3)=\{1,2,5,6\}$,
$L(w_3)=\{1,\ldots,6\}\setminus\psi'(y_3)$, $L(y_1)=\psi'(y_1)$, and let $L(u_1)$ be a $3$-element
subset of $\{3,4,5,6\}\setminus\psi'(u_0)$ containing the color $3$.  Let $R=\{x,v_1,u_1,v_2,u_2,v_3,u_3,w_1,w_3,y_1\}$,
and observe that $G[R]$ is isomorphic to the graph depicted in Figure~\ref{fig-H}.
Let $\psi$ be the union of the restriction of $\psi'$ to $G-R$ and the coloring of $G[R]$ obtained by Lemma~\ref{spe-set}
for the list assignment $L$.  Then $\psi$ is an $X$-enhanced coloring of $G$ extending $\varphi$, which is a contradiction.
\end{proof}

\begin{figure}[!htb]
\centering
{\includegraphics[height=0.3\textwidth]{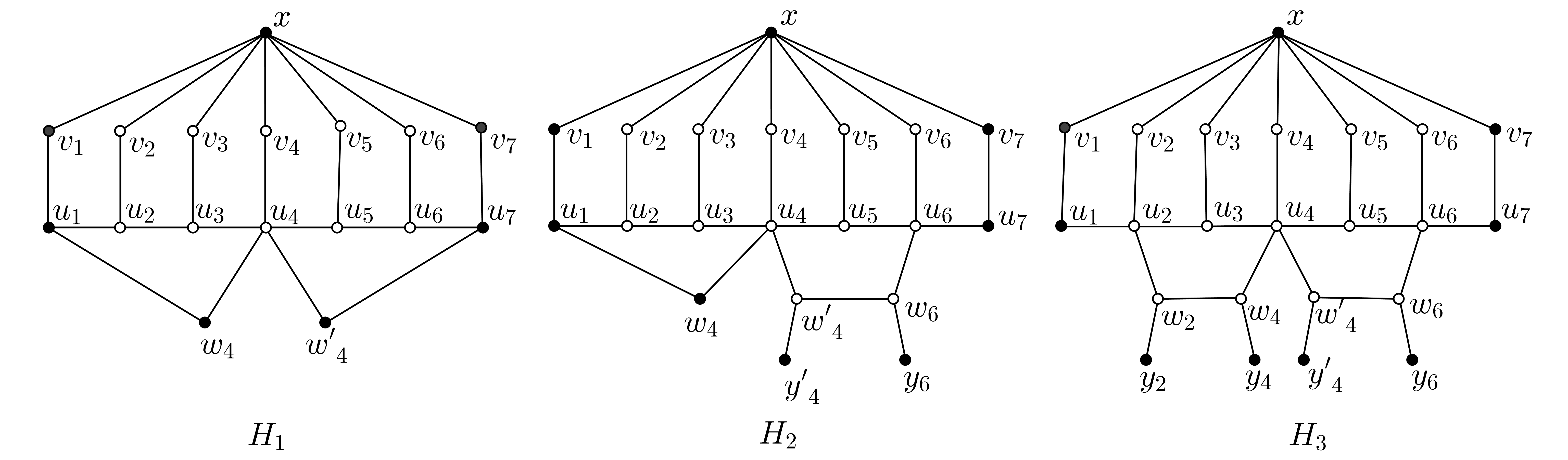}}

\caption{Subgraphs $H_1$, $H_2$ and $H_3$ from Lemma~\ref{232523}.  Vertices depicted by empty circles
are not incident with the outer face of $G$ and their degree in $G$ equals their degree in the figure.}\label{fig-Hi}
\end{figure}

Suppose $u_2u_3u_4w_4w_2$ is a cycle bounding a face in a plane graph $G$, where $u_2$, $u_3$, and $u_4$ are not incident
with the outer face, $\deg(u_3)=3$ and $\deg(u_4)=5$.  We say that the cycle is \emph{$(u_4,u_3)$-dangerous}
if either $\deg(u_2)=3$, or $\deg(u_2)=4$ and $\deg(w_4)=\deg(w_2)=3$.
We now exclude the situations in Figure~\ref{fig-Hi} involving dangerous faces.

\begin{lemma}
\label{232523}
Let $(G, X, \varphi)$ be a minimal counterexample with the outer face bounded by a cycle $C$ and with $X=\{x\}$.  Let $xv_2u_2u_3v_3$, $xv_3u_3u_4v_4$, $xv_4u_4u_5v_5$, and $xv_5u_5u_6v_6$ be distinct cycles bounding
$5$-faces in $G$, where $u_2,\ldots,u_6\not\in V(C)$ and $\deg(u_3)=\deg(u_5)=3$.  Let $K_1=u_2u_3u_4w_4w_2$
and $K_2=u_6u_5u_4w'_4w_6$ be $5$-cycles bounding faces.  If $\deg(u_4)=5$,
then $K_1$ is not $(u_4,u_3)$-dangerous or $K_2$ is not $(u_4,u_5)$-dangerous.
\end{lemma}
\begin{proof}
Note that $\deg(u_3)=\deg(u_5)=3$, $\deg(v_3)=\deg(v_5)=2$, and $u_2,u_6\not\in V(C)$, and thus $\deg(v_2)=\deg(v_6)=2$ and $v_2,v_6\notin V(C)$
by Lemma~\ref{233}.
By Corollary~\ref{5-face}, there exist paths $u_2u_1v_1x$ and $u_6u_7v_7x$ in $G$ with $u_1\neq u_3$ and $u_7\neq u_5$.
Suppose for a contradiction that $K_1$ is $(u_4,u_3)$-dangerous and $K_2$ is $(u_4,u_5)$-dangerous.
By Corollary~\ref{5-face}, Lemma~\ref{233}, and symmetry, we can assume that $G$ contains one of
the subgraphs $H_1$, $H_2$, or $H_3$ depicted in Figure~\ref{fig-Hi}
(up to possible identification of vertices $u_1$ and $u_7$ in the graph $H_3$; all other identifications
can be excluded using Lemma~\ref{basic}).
Let $G'$ be the graph obtained from $G-\{v_3,u_3,v_5,u_5\}$ by identifying $u_2$ and $w_4$,
and identifying $u_6$ and $w'_4$.
Using Lemma~\ref{basic}, observe $G'$ is triangle-free.  By the minimality of $G$, there exists an $X$-enhanced
coloring $\psi'$ of $G'$ extending $\varphi$. We may assume without lose of generality that $\psi'(x)=\{1,2,3\}$.

Suppose first that $\psi'(u_4)\subset \{1,2,3\}$, say $\psi'(u_4)=\{1,2\}$.
For $i\in \{2,6\}$, by Observation~\ref{int123} we have $\psi'(u_i)\cap\{1,2,3\}\neq\emptyset$.
Hence, we can assume $\psi'(u_2)=\psi'(w_4)=\{3,\alpha\}$ and $\psi'(u_6)=\psi'(w'_4)=\{3,\beta\}$
for some $\alpha,\beta\in \{4,5\}$.  Let $\psi(v)=\psi'(v)$ for
$v\in V(G)\setminus \{v_3,u_3,v_4,u_4,v_5,u_5\}$. Set $\psi(u_4)=\{2,6\}$, $\psi(v_4)=\{4,5\}$,
$\psi(u_3)=\{1,9-\alpha\}$, $\psi(v_3)=\{\alpha,6\}$,
$\psi(u_5)=\{1,9-\beta\}$, and $\psi(v_5)=\{\beta,6\}$.
Then $\psi$ is an $X$-enhanced coloring of $G$ extending $\varphi$, which is a contradiction.

Therefore, by Observation~\ref{int123}, we can assume $\psi'(u_4)\cap \{1,2,3\}=\{3\}$.  Let us now discuss the cases regarding
the ways $K_1$ and $K_2$ could be dangerous.
\begin{itemize}
\item[(i)] Suppose first that $\deg(u_2)=\deg(u_6)=3$, and thus $w_2=u_1$ and $w_6=u_7$, see the
subgraph $H_1$ in Figure~\ref{fig-Hi}.
Let $\psi$ be the restriction of $\psi'$ to $G'-\{v_2,u_2,v_6,u_6\}$.
If $\psi'(u_2)\neq \{1,2\}$, then we can set $\psi(u_2)=\psi'(u_2)$, $\psi(v_2)=\psi'(v_2)$,
choose $\psi(u_3)$ as a $2$-element subset of $\{1,\ldots,6\}\setminus(\psi'(u_2)\cup \psi'(u_4))$
containing color $1$ or $2$, and choose $\psi(v_3)$ as a $2$-element subset of $\{4,5,6\}\setminus\psi(u_3)$.
If $\psi'(u_2)=\{1,2\}$, then by Observation~\ref{int123} and symmetry,
we can assume $\psi'(u_1)=\{3,4\}$ and $6\not\in \psi'(u_4)$.
We set $\psi(u_2)=\{1,5\}$, $\psi(v_2)=\{4,6\}$,
$\psi(u_3)=\{2,6\}$ and $\psi(v_3)=\{4,5\}$.   Symmetrically, we extend $\psi$ to $u_5$, $v_5$, $u_6$, and $v_6$.
This gives an $X$-enhanced coloring of $G$ extending $\varphi$, which is a contradiction.

\item[(ii)] Hence, we can by symmetry assume that $\deg(u_6)=4$ and $w'_4$ and $w_6$ are vertices of degree three.
By Lemma~\ref{basic}, we have $w'_4,w_6\not\in V(C)$.
Suppose that $\deg(u_2)=3$, see the subgraph $H_2$ in Figure~\ref{fig-Hi}.
If $\psi'(u_6)\neq \{1,2\}$, then by Observation~\ref{int123},
we can assume that $\psi'(u_6)=\{2,4\}$ and $\psi'(u_4)\subseteq\{3,4,5\}$.
Let $\psi(v)=\psi'(v)$ for $v\in V(G)\setminus \{u_2,v_2,v_3,u_3,v_5,u_5\}$. Then $\psi$
extends to $u_2$, $v_2$, $u_3$, $v_3$ as in the previous case, and we can choose
$\psi(u_5)=\{1,6\}$ and $\psi(v_5)=\{4,5\}$.  This gives an $X$-enhanced coloring of $G$ extending $\varphi$, which is a contradiction.

Therefore, $\psi'(u_6)=\{1,2\}$.  By Observation~\ref{int123}, we can assume
$\psi'(u_7)=\{3,4\}$.  Let $R=\{x,v_4,u_4,v_5,u_5,v_6,u_6,w_4',y_4',w_6\}$,
where $y'_4$ is the neighbor of $w'_4$ distinct from $u_4$ and $w_6$.
Note that $G[R]$ is isomorphic to the graph depicted in Figure~\ref{fig-H}.
Since $3\in \psi'(u_4)$ and $\psi'(u_6)=\psi'(w'_4)$, by Observation~\ref{int123} we have $\psi'(w_4)\cap \{1,2\}\neq\emptyset$,
and thus there exists a $3$-element set $L(u_4)\subseteq \{3,\ldots,6\}\setminus \psi'(w'_4)$ containing the color $3$.
Let $\psi$ be the $X$-enhanced coloring of $G-\{u_3,v_3\}$ obtained from the restriction of $\psi'$ to $G-(R\cup\{u_3,v_3\})$
by extending it to $G[R]$ using Lemma~\ref{spe-set}.  Note that $\psi(u_4)\subseteq\{3,4,5,6\}$ and $3\in\psi(u_4)$.

By Observation~\ref{int123}, we have $\psi(u_2)\cap \{1,2\}\neq\emptyset$.
If $\psi(u_2)\neq\{1,2\}$, then $\psi$ can be extended to $u_3$ and $v_3$ by Observation~\ref{int123}. This gives an $X$-enhanced coloring of $G$
extending $\varphi$, which is a contradiction.
If $\psi(u_2)=\{1,2\}$, then by Observation~\ref{int123}, $\psi(u_1)=\{3,\alpha\}$ for some $\alpha\in \{4,5,6\}$.
Let $\psi_0(v)=\psi(v)$ for $v\in V(G)\setminus \{v_3,u_3,v_2,u_3\}$, $\psi_0(u_2)=\{1,\beta\}$, and $\psi_0(u_3)=\{2,\gamma\}$
for $\beta\in \{4,5,6\}\setminus \{\alpha\}$ and $\gamma\in\{4,5,6\}\setminus(\psi(u_4)\cup\{\beta\})$.
Then $\psi_0$ can be extended to $v_2$ and $v_3$ by Observation~\ref{int123}, giving an $X$-enhanced coloring of $G$
extending $\varphi$, which is a contradiction.

\item[(iii)] Therefore, $\deg(u_2)=\deg(u_6)=4$ and $w_2$, $w_4$, $w'_4$ and $w_6$ are vertices of degree three,
see the subgraph $H_3$ in Figure~\ref{fig-Hi}.
By Lemma~\ref{basic}, we have $w_2,w_4,w'_4,w_6\not\in V(C)$.
If $\psi'(u_2)\neq \{1,2\}$ and $\psi'(u_6)\neq \{1,2\}$, then $\psi'$
extends to $u_2$, $v_2$, $u_5$ and $v_5$ by Observation~\ref{int123}. This gives an $X$-enhanced coloring
of $G$ extending $\varphi$, which is a contradiction.

Hence, we can by symmetry assume $\psi'(u_6)=\{1,2\}$.  By Observation~\ref{int123}, we can assume $\psi'(u_7)=\{3,4\}$.
Let $R=\{x,v_4,u_4,v_5,u_5,v_6,u_6,w_4',y_4',w_6\}$, where $y'_4$ is the neighbor of $w'_4$ distinct from $u_4$ and $w_6$.
If $\psi'(u_2)\neq \{1,2\}$, then color $G[R]$ by Lemma~\ref{spe-set} and then extend the coloring to $u_3$ and $v_3$ as in the case (ii).
Hence, we can also assume that $\psi'(u_2)=\{1,2\}$, and $\psi'(u_1)=\{3,\alpha\}$ for some $\alpha\in\{4,5\}$.
Let $R'=\{x,v_2,u_2,v_3,u_3,v_4,u_4,w_4,y_4,w_2\}$, where $y_4$ is the neighbor of $w_4$ distinct from $u_4$ and $w_2$.
Since $\psi'(w_4)=\psi'(w'_4)=\{1,2\}$, we have $\psi'(y_4),\psi'(y'_4)\subset \{3,4,5,6\}$, and if $\psi'(y_4)\cap \psi'(y'_4)=\emptyset$,
then $\psi'(y_4)\cup\psi'(y'_4)=\{3,4,5,6\}$.  Hence, there exists a $2$-element set $S\subseteq \{3,4,5,6\}$ such that
$3\in S$ and $S\cap \psi'(y_4)\neq\emptyset\neq S\cap\psi'(y'_4)$.

Let $\psi$ be the restriction of $\psi'$ to $G-(R\cup R')$.  By Lemma~\ref{spe-set}, $\psi$ extends to
colorings $\psi_1$ of $G[R]$ and $\psi_2$ of $G[R']$ such that $\psi_1(x)=\psi_2(x)=\{1,2,3\}$ and
$\psi_1(u_4)=\psi_2(u_4)=S$.  Note also $\psi_1(v_4)=\psi_2(v_4)=\{3,4,5,6\}\setminus S$.  Then $\psi\cup\psi_1\cup\psi_2$ is an $X$-enhanced coloring of $G$ extending $\varphi$,
which is a contradiction.
\end{itemize}
\end{proof}

\subsection{Discharging}

\subsubsection{Notation}

Consider a minimal counterexample $(G, X, \varphi)$ with the outer face bounded by a cycle $C$. By Corollary~\ref{5-face},
every face other than the outer one is a 5-face. If $X=\{x\}$, consider a cycle
$xv_1u_1u_2v_2$ such that $u_1,v_1\not\in V(C)$ bounding a $5$-face $f$.
If $\deg(v_1)=2$, $\deg(u_1)=3$ and $u_2\notin V(C)$, then $\deg(v_2)=2$ and $v_2\notin V(C)$ by Lemma~\ref{233}, and we say
$f$ is a {\em type-A} face. By Lemma~\ref{basic}, we have $\deg(u_2)\geq 3$. If
$\deg(u_2)= 3$, then we say $f$ is a {\em type-A-1} face incident with $x$. If
$\deg(u_2)= 4$, then we say $f$ is a {\em type-A-2} face incident with $x$. If
$\deg(u_2)\geq 5$, then we say $f$ is a {\em type-A-3 face} incident with $x$.
Since $\deg(v_1)=\deg(v_2)=2$, there exists cycles $xv_1u_1u_0v_0$ and $xv_2u_2u_3v_3$ bounding faces distinct
from $f$. For $i\in\{0,3\}$, if $u_i\in V(C)$ or $\deg(u_i)\geq 5$, and $f$ is a type-A-1 face or type-A-2 face,
then we say $u_i$ is \emph{connected} to $f$.

Suppose $f$ is a $5$-face bounded by a 5-cycle $xv_1u_1u_2v_2$ satisfying $u_1,v_1\not\in V(C)$,
$\deg(v_1)=2$, $\deg(u_1)=4$, and $\deg(v_2)\ge 3$ ($v_2$ may or may not belong to $V(C)$).  In this case we say $f$ is a {\em type-B face}.

Suppose now a cycle $xv_1u_1u_2v_2$ bounds a type-A-2 face $f_1$ incident with $x$, where $\deg(u_1)=4$ and $\deg(u_2)=3$.
Since $\deg(v_2)=2$, there exists a cycle $xv_2u_2u_3v_3$ bounding a face $f_2$ distinct from $f_1$.
Suppose furthermore $u_3\not\in V(C)$; then $\deg(v_3)=2$ by Lemma~\ref{233}, and $\deg(u_3)\geq 4$ by Lemma~\ref{232325}. Let us
consider the case that $\deg(u_3)=4$, and let $u_1u_2u_3w_3w_1$ be the cycle bounding the $5$-face $g$ incident with $u_2$ distinct from $f_1$ and $f_2$.
Note that $\deg(w_1)\geq 4$ or $\deg(w_3)\geq 4$ by Lemma~\ref{242324}. We say $g$ is a {\em type-C face}, and
for $i\in\{1,2\}$ we say $g$ is \emph{connected} to $f_i$ if $\deg(w_{2i-1})=3$.
Note that a type-C face is connected to at most one type-A-2 face and is incident with
at least three $4^+$-vertices.
A type-A-2 face is \emph{tight} if no vertex or type-C face is connected to it.

Continuing in the situation of the previous paragraph,
suppose that $\deg(w_3)\geq 4$. Since $\deg(v_3)=2$, there exists a cycle
$xv_3u_3u_4v_4$ bounding a $5$-face $f_3$ distinct from $f_2$.
Suppose that $u_4\not\in V(C)$.
\begin{itemize}
\item If $\deg(v_4)=2$, then
$\deg(u_4)\geq 4$ by Lemma~\ref{232424}. If $\deg(u_4)=4$, then let
$w_3u_3u_4w_4y_3$ be a cycle bounding the $5$-face $h$ incident with $u_3$
distinct from $f_2$, $f_3$, and $g$, see the left graph in Figure~\ref{fig-type}
for an illustration. We say $h$ is a {\em type-D face} connected to $f_2$.
\item Suppose now $\deg(v_4)=\deg(u_4)=3$ (so $f_3$ is a type-B face) and a cycle
$v_4u_4w_4y_4z_4$ bounds a $5$-face $k\neq f_3$,
where $\deg(y_4)=\deg(z_4)=3$, $v_4,w_4,y_4,z_4\not\in V(C)$
and $\deg(w_4)\geq 4$. Let $q$ be the face incident with $u_4$ distinct from $f_3$ and $k$,
bounded by the cycle $q=w_4u_4u_3w_3y_3$, see the right graph in Figure~\ref{fig-type}
for an illustration. We say $q$ is a {\em type-E face} connected to $k$.  Note that each type-E face is incident with
at least three $4^+$-vertices.
\end{itemize}
By Lemma~\ref{basic}, the distance of $w_3$ and $w_4$ from $x$ is three,
and thus a type-D or type-E face cannot also be a type-A, type-B, or type-C face,
and a type-D face cannot also be a type-E face.
Furthermore, each type-D face is connected to $p\le 2$ type-A-2
faces and is incident with at least $(p+2)$ $4^+$-vertices,
and each type-E face is connected to a unique face.

\begin{figure}[!htb]
\centering
{\includegraphics[height=0.3\textwidth]{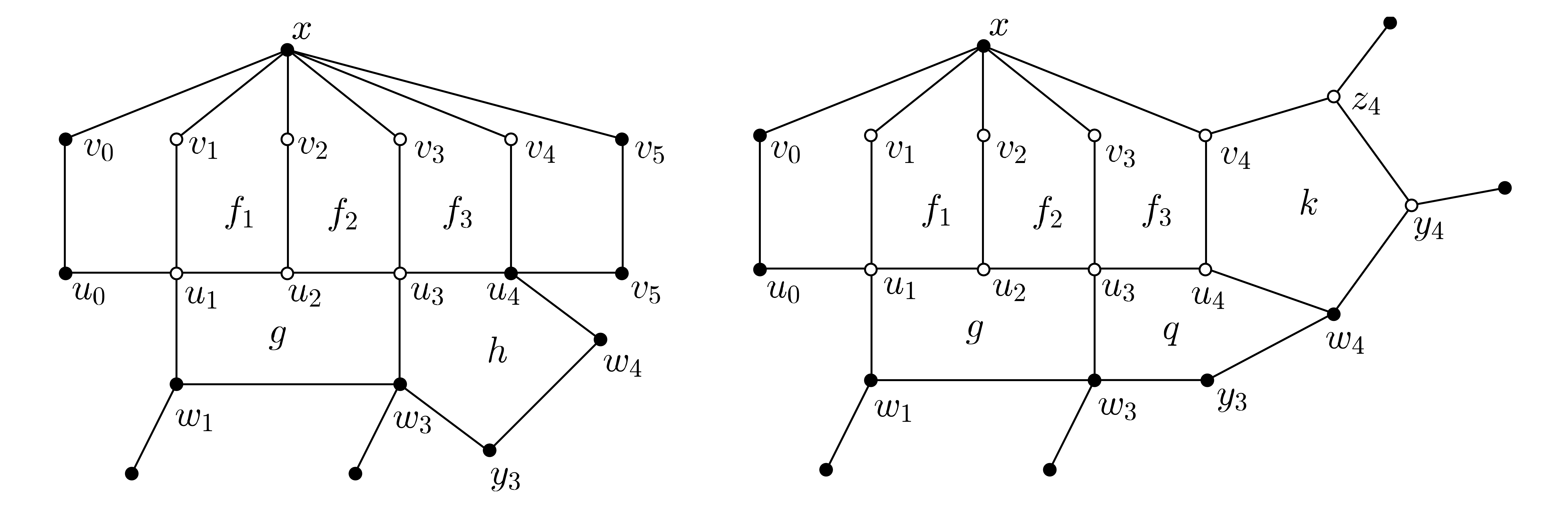}}

\caption{Type-C face $g$, type-D face $h$, and type-E face $q$.}\label{fig-type}
\end{figure}

Suppose now cycles $xv_1u_1u_2v_2$ and $xv_2u_2u_3v_3$ bound distinct $5$-faces $f_1$
and $f_2$, where $u_1,u_2,u_3\not\in V(C)$, $\deg(v_1)=\deg(v_2)=\deg(v_3)=2$,
$\deg(u_1)=5$, $\deg(u_2)=3$, and $\deg(u_3)=4$.  Let $g$ be the face incident with $u_2$
distinct from $f_1$ and $f_2$, bounded by the cycle $u_1w_1w_3u_3u_2$.
If for some $i\in\{1,3\}$, the vertex $w_i$ has degree at least four,
we say $g$ is a {\em type-F face} connected to $u_1$. Note that each type-F face is incident with
at most two vertices of degree three not belonging to $V(C)$.  By Lemma~\ref{basic}, the distance of $w_3$ and $w_4$ from $x$ is three,
and thus a type-F face cannot also be a type-A, \ldots, or type-E face, and each type-F face is connected
to a unique vertex.

Let $Q$ be a 5-cycle in $G$ vertex-disjoint from $X$ and intersecting $C$ in at most one vertex.  We say the
face bounded by $Q$ is \emph{tied} to a vertex $z\in V(C)$ if $z\notin V(Q)$
and $z$ has a neighbor in $V(Q)\setminus V(C)$ of degree three. Suppose $X=\{x\}$
and $x$ is tied to a 5-face $f$ not incident with $x$ bounded by
the cycle $v_5v_1v_2v_3v_4$ via an edge $xv_5$.  By Lemmas~\ref{basic} and \ref{4-face},
no vertex of $C$ is incident with $f$.  By Corollary~\ref{x-tie-5-face}, a vertex
incident with $f$ has degree at least four, without loss of generality $v_1$ or $v_2$.
If four vertices of $Q$ have degree three, then let $g$ be the face whose boundary
contains the path $xv_5v_1$; in this situation, we say that $f$ is a \emph{special 5-face tied to $x$ and
connected to $g$}.

\subsubsection{Initial charge and discharging rules}

Now we proceed by the discharging method. Consider a minimal counterexample $(G,X,\varphi)$ with the outer
face bounded by the cycle $C$.
Set the initial charge of every vertex $v$ of $G$ to be $\ch_0(v)=\deg (v)-4$,
and the initial charge of every face $f$ of $G$ to be $\ch_0(f)=|f|-4$.
By Euler's formula,
\begin{align}
\sum_{v\in V(G)}\ch_0(v)+\sum_{f\in F(G)}\ch_0(f)&=\sum_{v\in V(G)}(\deg (v)-4)+\sum_{f\in F(G)}(|f|-4)\nonumber\\
&=4(|E(G)|-|V(G)|-|F(G)|)=-8.\label{eq:sum}
\end{align}

We can without loss of generality assume that $X\neq\emptyset$ (and thus $|X|=1$),
as otherwise we observe that the cycle $C$ bounding the outer face contains a subpath
$uxv$ such that $|\varphi(u)\cup\varphi(v)\cup\varphi(x)|\le 5$, and we can
set $X=\{x\}$ and add a color to $\varphi(x)$.  Let $x$ denote the unique vertex in $X$.
We redistribute the charges according to the following rules.

\begin{enumerate}[label=\textbf{R\arabic*}]
\item Each face other than the outer one sends $\frac{1}{3}$ to each incident vertex that either has degree two and belongs to $V(C)$, or has degree three and does not belong to $V(C)$. \label{ch:fto3ver}
\item Each face sends $1$ to each incident vertex that has degree two and does not belong to $V(C)$.  \label{ch:fto2ver}
\item The vertex $x$ sends 1 to each incident face other than the outer one.  \label{ch:xtof}
\item Each $5^+$-vertex other than $x$ sends $\frac{1}{3}$ to each incident type-A-3 face.  \label{ch:5vtoAdirect}
\item If $v\neq x$ is a $5^+$-vertex or belongs to $V(C)$, then $v$ sends $\frac{1}{3}$ to each connected type-A-1 face or type-A-2 face.  \label{ch:5vtoAconnect}
\item Each type-B face sends $\frac{1}{3}$ to each tight type-A-2 face with which it shares an edge incident with $x$. \label{ch:BtoA}
\item Each type-C face sends $\frac{1}{3}$ to each connected type-A-2 face.  \label{ch:CtoA}
\item Each type-D face sends $\frac{1}{3}$ to each connected type-A-2 face. \label{ch:DtoA}
\item Each type-F face sends $\frac{1}{3}$ to each connected 5-vertex. \label{ch:Fto5v}
\item Suppose $f$ is a special $5$-face tied to $x$ and connected to a face $g$. If a type-E face $h$ is connected to $f$,
then $h$ sends $\frac{1}{3}$ to $f$, otherwise $g$ sends $\frac{1}{3}$ to $f$.\label{ch:xto5face}
\item Each vertex on the outer face other than $x$ sends $\frac{1}{3}$ to each 5-face tied to it.\label{ch:vto5face}
\end{enumerate}

Let the charge obtained by these rules be called final and denoted by $\ch$.  Note that the redistribution does not change the
total amount of charge, and thus the sum of the final charges assigned to vertices and faces of $G$ is $-8$ by (\ref{eq:sum}).

\subsubsection{Final charges of vertices}

\begin{lemma}
\label{inter-vertex}
Let $(G, \{x\}, \varphi)$ be a minimal counterexample with the outer face bounded by a cycle $C$.
Then each vertex $v \in  V (G)\setminus V(C)$ satisfies $\ch(v)\geq 0$.
\end{lemma}

\begin{proof}
By Lemma~\ref{basic}, $v$ has degree at least two. If $v$ has degree two, then $v$ receives 1 from
each incident face by \ref{ch:fto2ver}, and thus $\ch(v)=\ch_0(v)+2\times 1=0$. If $v$ has degree three, then it
receives $\frac{1}{3}$ from each incident face by \ref{ch:fto3ver}, and thus $\ch(v) = \ch_0(v)
+3\times \frac{1}{3} = 0$. If $v$ has degree 4, then $\ch(v)=\ch_0(v)=0$.

If $v$ has degree five, then $v$ sends $\frac{1}{3}$ to each incident type-A-3 face
by \ref{ch:5vtoAdirect}, and each connected type-A-1 face or type-A-2 face by
\ref{ch:5vtoAconnect}. Let $k$ be the number of faces to that $v$ sends charge.
By Lemma~\ref{4-face}, there exists at most one path of length two between $v$ and $x$, and
thus $v$ is incident with at most two type-A-3 faces, and connected to at most two type-A-1 or type-A-2 faces,
implying that $k\le 4$.
If $k\leq 3$, then $\ch(v)\geq \ch_0(v) -3\times \frac{1}{3} = 0$.
Hence, we can assume $k=4$.
By Lemma~\ref{232523},  $v$ is incident with at least one type-F face, from which it receives $1/3$ by \ref{ch:Fto5v}.
Therefore, $\ch(v) \geq \ch_0(v)-4\times \frac{1}{3}+\frac{1}{3}= 0$.

If $\deg(v)\geq 6$, then similarly $v$ sends charge to at most four faces by \ref{ch:5vtoAdirect}
and \ref{ch:5vtoAconnect}, and $\ch(v)\geq \ch_0(v)-4\times \frac{1}{3}>0$.
\end{proof}

\begin{lemma}
\label{C-vertex}
Let $(G, \{x\}, \varphi)$ be a minimal counterexample with the outer face bounded by a cycle $C$.
Then $\ch(x)=-3$, and for any vertex $v\in V(C)\setminus\{x\}$,
$\ch(v)=-\frac{5}{3}$ if $\deg(v)=2$ and $\ch(v)\geq \frac{2}{3}(\deg(v)-5)$ if $\deg(v)\geq 3$.
\end{lemma}

\begin{proof}
Note that $x$ sends 1 to each incident face other than the outer one by \ref{ch:xtof}, and
thus $\ch(x)=\ch_0(x)-(\deg(v)-1)=-3$.  Consider a vertex $v\in V(C)\setminus \{x\}$.
If $\deg(v) = 2$, then $v$ receives $\frac{1}{3}$ from the incident non-outer face by \ref{ch:fto3ver} and
$\ch(v) = \ch_0(v)+\frac{1}{3}=-\frac{5}{3}$. If $\deg(v)\ge 3$, then by \ref{ch:5vtoAdirect} and \ref{ch:vto5face}
$v$ sends $\frac{1}{3}$ to at most $\deg(v)-2$ faces tied or connected to it, and thus
$\ch(v) \geq \ch_0(v)-(\deg(v)-2)\times \frac{1}{3} = \frac{2}{3}(\deg(v)-5)$.
\end{proof}

\subsubsection{Final charges of faces}

\begin{lemma}
\label{interface}
Let $(G, \{x\}, \varphi)$ be a minimal counterexample with the outer face bounded by a cycle $C$.  Every face $f$ not incident with $x$ satisfies $\ch(f)\geq 0$.
\end{lemma}

\begin{proof}
By Corollary~\ref{5-face}, we have $|f|=5$ and $\ch_0(f)=1$.
Since $f$ is not incident with $x$, Lemma~\ref{basic} implies that
every vertex of degree two incident with $f$ belongs to $V(C)$, and thus $f$
does not send charge by \ref{ch:fto2ver}.  By \ref{ch:fto3ver},
$f$ sends at most $\frac{1}{3}$ to each incident vertex.

If $f$ is a type-C
face, then $f$ sends $\frac{1}{3}$ to each connected type-A-2 face by
\ref{ch:CtoA}. Recall that $f$ is connected to at most one type-A-2 face and
$f$ is incident with at least three $4^+$-vertices, i.e., the number of vertices to that $f$
sends charge is at most 2. Then $\ch(v)\geq \ch_0(v)-\frac{1}{3}-2\times
\frac{1}{3}=0$.

If $f$ is a type-D face, then $f$ sends $\frac{1}{3}$ to each connected
type-A-2 face by \ref{ch:DtoA}. Suppose that $f$ is connected to $p$  type-A-2
faces.  Recall that $p\leq 2$ and $f$ is incident with at least $(p+2)$ $4^+$-vertices, and thus the
number of vertices to that $f$ sends charge is at most $5-(2+p)=3-p$.  Hence, $\ch(v)\geq
\ch_0(v)-p\times\frac{1}{3}-(3-p)\times \frac{1}{3}=0$.

If $f$ is a type-E face, then $f$ is connected to exactly one special 5-face $g$
tied to $x$, and $f$ sends $\frac{1}{3}$ to $g$ by \ref{ch:xto5face}. Recall that
$f$ is incident with least three $4^+$-vertices, and thus the number of
vertices to that $f$ sends charge is at most 2. Hence, $\ch(v)\geq
\ch_0(v)-\frac{1}{3}-2\times \frac{1}{3}=0$.

If $f$ is a type-F face, then $f$ is connected to exactly one 5-vertex $v$, and $f$ sends $\frac{1}{3}$ to $v$ by \ref{ch:Fto5v}.
Recall that the number of vertices to that $f$
sends charge is at most 2. Then $\ch(v)\geq \ch_0(v)-\frac{1}{3}-2\times
\frac{1}{3}=0$.

Therefore, $f$ is not a type-C, type-D, type-E, or type-F face.
Hence,  $f$ only sends $\frac{1}{3}$ to each incident 2-vertex in $C$ or 3-vertex not in $C$
by \ref{ch:fto3ver}.
Let $k$ be the number of vertices to that $f$ sends charge. If $k\leq 3$, then
$\ch(f)\geq 0$, and thus we can assume that $k \geq 4$. If $f$ is incident with
a vertex $v$ of degree two, then note that $v \in V (C)$ by Lemma~\ref{basic}.
Furthermore, since $G$ is 2-connected and $G \neq C$, we conclude that $f$ is
incident with at least two $3^+$-vertices belonging to $V(C)$, to
which $f$ does not send charge. This contradicts the assumption that $k \geq  4$.
Hence, no vertex of degree two is incident with $f$, and thus $k$ is the number
of incident vertices of degree three not belonging to $V(C)$.
If $f$ is tied to $x$, then $f$ is incident with exactly four 3-vertices by Corollary~\ref{x-tie-5-face}. By
\ref{ch:xto5face}, $f$ receives $\frac{1}{3}$ from some face, and thus
$\ch(f) = \ch_0(f)-4\times \frac{1}{3} + \frac{1}{3}= 0$. If $f$ is not
tied to $x$, then $f$ is tied to at least $k-3$ vertices of $C$ by
Lemma~\ref{tie-5-face} and $f$ receives $\frac{1}{3}$ from each of them
by \ref{ch:vto5face}, and $\ch(f)\geq \ch_0(f)-k\times \frac{1}{3} + (k-3)\times \frac{1}{3}= 0$.
\end{proof}

\begin{lemma}
\label{inter-x-face}
Let $(G, \{x\}, \varphi)$ be a minimal counterexample with the outer face bounded by a cycle $C$.
Any face $f$ incident with $x$ other than the outer one satisfies $\ch(f)\geq 0$.
\end{lemma}

\begin{proof}
By Corollary~\ref{5-face}, we have $|f|=5$ and $\ch_0(f)=1$.
Note that $f$ receives $1$ from $x$ by \ref{ch:xtof} and sends charge only by
\ref{ch:fto3ver}, \ref{ch:fto2ver}, \ref{ch:BtoA}, and \ref{ch:xto5face}.
Let $xv_3u_3u_4v_4$ denote the cycle bounding $f$.

Consider first the case that neither $v_3$ nor $v_4$ is a vertex of degree two not belonging to $C$.
Then $f$ sends at most $4\times \frac{1}{3}$ by \ref{ch:fto3ver} and at most $2\times \frac{1}{3}$ by \ref{ch:xto5face},
implying $\ch(f)\ge\ch_0(f)+1-4\times \frac{1}{3}-2\times \frac{1}{3}=0$.

Hence, we can assume $\deg(v_3)=2$ and $v_3\not\in V(C)$, and thus $f$ sends $1$ to $v_3$ by \ref{ch:fto2ver}.
By Lemma~\ref{4-face}, we have $u_3\notin V(C)$, and thus $\deg(u_3)\geq 3$ by Lemma~\ref{basic}.
Let $f_2\neq f$ be the other $5$-face incident with $xv_3$, bounded by a cycle $xv_3u_3u_2v_2$. By \ref{ch:fto3ver} and \ref{ch:BtoA},
$f$ sends at most $\frac{1}{3}$ to $u_3$ and $f_2$ in total.
We now discuss the case that $v_4$ is not a vertex of degree two not belonging to $C$.
\begin{itemize}
\item If $v_4\in V(C)$, then $f$ does not send charge by \ref{ch:xto5face}, and sends at most $\frac{1}{3}$ to
$v_4$ and $u_4$ in total by \ref{ch:fto3ver} (if $\deg(v_4)=2$, then $u_4\in V(C)$, and since $u_3\not\in V(C)$, we have $\deg(u_4)\ge 3$).
Hence, $\ch(f)\ge\ch_0(f)+1-1-\frac{1}{3}-\frac{1}{3}>0$.
Therefore, we can assume $v_4\notin V(C)$, and thus $\deg(v_4)\ge 3$.
By Lemma~\ref{4-face}, we have $u_4\notin V(C)$. Then $\deg(u_3)\geq 4$ by Lemma~\ref{233}, and $f$ does not send charge to $u_3$ by \ref{ch:fto3ver}.
\item If $f$ sends at most $\frac{1}{3}$ by \ref{ch:BtoA} and \ref{ch:xto5face} in total, or $f$ does not send charge by \ref{ch:fto3ver} to at least
one of $v_4$ and $u_4$, then $\ch(f)\ge\ch_0(f)+1-1-3\times \frac{1}{3}=0$.
\item Hence, we can assume that $f$ sends charge by both \ref{ch:BtoA} and \ref{ch:xto5face}, and $f$ sends charge to both $v_4$ and $u_4$ by \ref{ch:fto3ver}.
Consequently, $\deg(u_3)=4$, $\deg(v_4)=\deg(u_4)=3$, and
the neighbor $w_4$ of $u_4$ distinct from $u_3$ and $v_4$ is a $4^+$-vertex. Since $f$ sends charge by \ref{ch:BtoA}, we conclude that $f_2$ is a tight type-A-2 face, and thus $\deg(v_2)=2$, $\deg(u_2)=3$ and $v_2,u_2\not\in V(C)$.
Let $f_1\neq f_2$ be the other $5$-face incident with $xv_2$, bounded by a cycle $xv_1u_1u_2v_2$. Since $f_2$ is tight, we have $u_1\not\in V(C)$ and $\deg(u_1)\le 4$, and thus $v_1\not\in V(C)$
and $\deg(v_1)=2$ by Lemma~\ref{233}.  By Lemma~\ref{232325}, we have $\deg(u_1)=4$.
Let $u_1u_2u_3w_3w_1$ be the cycle
bounding the face $g$ incident with $u_2$ distinct from $f_1$ and $f_2$.  Since $f_2$ is tight,
we conclude that $w_3$ is a $4^+$-vertex, and thus the face incident with $u_3$ distinct from $f$, $f_2$, and $g$
is a type-E face, contradicting the assumption that $f$ sends charge by \ref{ch:xto5face}.
\end{itemize}

Finally, let us consider the case that both $v_3$ and $v_4$ are vertices of degree two not belonging to $V(C)$.
Then $f$ sends charge only by \ref{ch:fto3ver} and \ref{ch:fto2ver}.  By Lemma~\ref{basic}, we have $u_3,u_4\not\in V(C)$.
If both $u_3$ and $u_4$ are $4^+$-vertices, then $f$ only sends charge to $v_3$ and $v_4$ and $\ch(f)\ge\ch_0(f)+1-2\times 1=0$.
Therefore, we can assume $\deg(u_3)=3$.  If $u_4$ is a $5^+$-vertex, then $f$ receives $\frac{1}{3}$ from $u_4$ by \ref{ch:5vtoAdirect}
and $\ch(f)\ge\ch_0(f)+1-2\times 1-\frac{1}{3}+\frac{1}{3}=0$, and thus we can assume $\deg(u_4)\le 4$.
Let $f_2\neq f$ be the face incident with $xv_3$, bounded by a cycle $xv_2u_2u_3v_3$, and
let $f_4\neq f$ be the face incident with $xv_4$, bounded by a cycle $xv_4u_4u_5v_5$.

If $\deg(u_4)=3$, then by Lemma~\ref{232325}, we have either $u_i\in V(C)$ or $\deg(u_i)\ge 5$ for $i\in\{2,5\}$,
and $f$ receives $2\times \frac{1}{3}$ by \ref{ch:5vtoAconnect}, and
$\ch(f)\ge\ch_0(f)+1-2\times 1-2\times\frac{1}{3}+2\times \frac{1}{3}=0$.
Therefore, we can assume $\deg(u_4)=4$.  If $u_2\in V(C)$ or $\deg(u_2)\ge 5$, then $f$ receives $\frac{1}{3}$ by \ref{ch:5vtoAconnect}
and $\ch(f)\ge\ch_0(f)+1-2\times 1-\frac{1}{3}+\frac{1}{3}=0$.  Hence, we can assume $u_2\not\in V(C)$ and $\deg(u_2)\le 4$,
and analogously, $u_5\not\in V(C)$ and $\deg(u_5)\le 4$. By Lemma~\ref{233}, $\deg(v_2)=2$, and $\deg(u_2)=4$
by Lemma~\ref{232325}.  Let $u_2u_3u_4w_4w_2$ be the cycle
bounding the face $g$ incident with $u_3$ distinct from $f_2$ and $f$.  If $\deg(w_4)=3$, then
$f$ receives $\frac{1}{3}$ by \ref{ch:CtoA}, and $\ch(f)\ge\ch_0(f)+1-2\times 1-\frac{1}{3}+\frac{1}{3}=0$.
Hence, we can assume $w_4$ is a $4^+$-vertex, which implies $f$ is a tight type-A-2 face. If $\deg(v_5)\geq 3$,  then $f_4$ is a type-B face. By \ref{ch:BtoA}, $f$ receives $\frac{1}{3}$ from $f_4$, and $\ch(f)\ge\ch_0(f)+1-2\times 1-\frac{1}{3}+\frac{1}{3}=0$.
Therefore, $\deg(v_5)=2$, and since $u_5\not\in V(C)$, we have $v_2\notin V(C)$.
By Lemma~\ref{232424}, we have $\deg(u_5)=4$.  However, then $f$ receives $\frac{1}{3}$ by \ref{ch:DtoA},
and $\ch(f)\ge\ch_0(f)+1-2\times 1-\frac{1}{3}+\frac{1}{3}=0$.
\end{proof}

\subsubsection{Proof of Theorem~\ref{X-tri-free}}

\begin{proof}[Proof of Theorem~\ref{X-tri-free}]
Suppose for a contradiction there exists a minimal counterexample $(G, X, \varphi)$,
with the outer face bounded by a cycle $C$.  As we argued before, we can assume $X\neq\emptyset$; let $X=\{x\}$.
By Lemma~\ref{inter-vertex},
$\ch(v)\geq 0$ for $v\in V(G)\setminus V(C)$. By Lemmas~\ref{interface} and \ref{inter-x-face}, $\ch(f)\geq 0$ for every non-outer face $f$ of $G$.

The final charge of the outer face is $|C|-4$. Consider a vertex $v \in V (C)$.
By Lemma~\ref{C-vertex}, $\ch(x)=-3$, $\ch(v) =-\frac{5}{3}$ if
$v\neq x$ and $\deg(v) = 2$, and $\ch(v) \geq \frac{2}{3}(\deg(v)-5)\geq -\frac{4}{3}$ if $v\neq x$ and
$\deg(v)\geq 3$.   If $|C|=4$, then by Lemma~\ref{basic} and Corollary~\ref{5-face}, all vertices of $C$ have degree
at least three, and thus the sum of the final charges is at least
$(|C|-4)-3-3\times \frac{4}{3}=-7$, a contradiction to (\ref{eq:sum}).

Therefore, $|C|=5$; let $C=xv_1v_2v_3v_4$.  If $V(C)\setminus\{x\}$ contains at most one vertex of degree two, then
the sum of the final charges is at least $(|C|-4)-3-\frac{5}{3}-3\times \frac{4}{3}=-\frac{23}{3}>-8$, a contradiction to (\ref{eq:sum}).
By Lemma~\ref{4-face} and Corollary~\ref{5-face}, no two vertices of degree two in $V(C)\setminus\{x\}$ are adjacent.
Hence, exactly two vertices of $V(C)\setminus\{x\}$ have degree two.  If a vertex of $V(C)\setminus\{x\}$ has degree at least $4$,
then the sum of the final charges is at least $(|C|-4)-3-2\times \frac{5}{3}-\frac{4}{3}-\frac{2}{3}=-\frac{22}{3}>-8$,
a contradiction to (\ref{eq:sum}).  Hence, we can by symmetry assume that $\deg(v_1)=2$ and either $\deg(v_3)=2$ or $\deg(v_4)=2$,
and all other vertices of $V(C)\setminus\{x\}$ have degree exactly three.

If $\deg(v_1)=\deg(v_3)=2$, then by Corollary~\ref{5-face} $G$ contains cycles $v_2y_2y_4v_4v_3$
and $v_2y_2y'_4xv_1$ bounding $5$-faces.  However, then Lemma~\ref{basic} applied to the cycle $y_2y_4v_4xy'_4$ implies
$\deg(y_4)=2$, which is a contradiction.

If $\deg(v_1)=\deg(v_4)=2$, then by Lemma~\ref{basic} and Corollary~\ref{5-face}, $G$ contains cycles
$v_1v_2y_2x_2x$, $v_4v_3y_3x_3x$, and $v_2v_3y_3yy_2$ bounding $5$-faces $f_1$, $f_3$, and $f$.
If $\deg(x_2)\ge 3$, then $f_1$ sends at most $3\times\frac{1}{3}$ to $v_1$, $y_2$, and $x_2$ by \ref{ch:fto3ver}
and at most $\frac{1}{3}$ by \ref{ch:xto5face} and receives $1$ from $x$ by \ref{ch:xtof},
implying $\ch(f_1)\ge \ch_0(f_1)+1-3\times\frac{1}{3}-\frac{1}{3}=\frac{2}{3}$.
It follows that the sum of the final charges is at least $(|C|-4)-3-2\times\frac{5}{3}-2\times \frac{4}{3}+\frac{2}{3}=-\frac{22}{3}>-8$,
a contradiction.  Consequently, $\deg(x_2)=2$. Then $f_1$ sends $1$ to $x_2$ by \ref{ch:fto2ver} and at most $2\times \frac{1}{3}$ to
$v_1$ and $y_2$ by \ref{ch:fto3ver}, and does not send anything by \ref{ch:xto5face}, implying
$\ch(f_1)\ge \ch_0(f_1)+1-1-2\times\frac{1}{3}=\frac{1}{3}$.
Then, the sum of the final charges is at least $(|C|-4)-3-2\times\frac{5}{3}-2\times \frac{4}{3}+\frac{1}{3}=-\frac{23}{3}>-8$,
which is again a contradiction.

We conclude there exists no counterexample to Theorem~\ref{X-tri-free}.
\end{proof}

\section{Set coloring of planar graphs of girth at least 5}

\subsection{Strong hyperbolic property}

A class $G$ of graphs embedded in closed surfaces
(which possibly can have a boundary) is {\em hyperbolic} if there exists a constant $c_{\mathcal{G}}$
such that for each graph $G\in \mathcal{G}$ embedded in a surface $\Sigma$ and each open disk
$\Lambda\subset\Sigma$ whose boundary $\partial\Lambda$ intersects $G$ only in vertices, the number of vertices
of $G$ in $\Lambda$ is at most $c_{\mathcal{G}}(|\partial\Lambda\cap G|-1)$. The class is {\em strongly hyperbolic} if the same
holds for all sets $\Lambda\subset\Sigma$ homeomorphic to an open cylinder (sphere with two
holes).

Let $G$ be a graph and let $S$ be a proper subgraph of $G$.
We say $G$ is {\em $S$-critical} for $(6:2)$-coloring if for
every proper subgraph $H \subset G$ such that $S \subseteq H$, there exists a
$(6:2)$-coloring of $S$ that extends to a $(6:2)$-coloring of $H$, but not to a
$(6:2)$-coloring of $G$.

In~\cite{listfraccyl}, we proved a strengthening of the following claim.

\begin{theorem}[Dvo\v{r}\'ak and Hu~\cite{listfraccyl}]\label{cylinder}
Let $\GG$ be the class of graphs of girth at least five embedded in surfaces such that if $G\in\GG$
is embedded in $\Sigma$ and $S$ is the subgraph of $G$ drawn in the boundary of $\Sigma$,
then $G$ is $S$-critical for $(6:2)$-coloring.  Then $\GG$ is strongly hyperbolic.
\end{theorem}

By Theorem~7.11 in \cite{PosThoHyperb}, we have the following result.

\begin{theorem}\label{Ckfraccri}
There exists a constant $\lambda>0$ such that the following holds.
Let $G$ be a plane graph of girth at least five and let $C_1$, \ldots, $C_k$ be cycles bounding faces of $G$.
If $G$ is $(C_1\cup C_2\cup \cdots \cup C_k)$-critical for $(6:2)$-coloring, then $|V(G)|\leq \lambda\sum_{i=1}^{k}|C_i|$.
\end{theorem}

\subsection{Proof of Theorem~\ref{X-girth5}}

Let $G$ be a plane graph of girth at least 5 and let $x$ be a vertex of $G$ with neighbors $y_1$,\ldots, $y_d$ in order.
For $d\ge 3$, to \emph{split} $x$ is to replace $x$ by $d$ independent vertices $y^{1}, y^{2}, \ldots, y^{d}$, and to replace each edge $xy_i$ by edges $y_iy^{i}$ and $y_iy^{i+1}$ (where $y^{d+1}=y^1$). Then $C_x=y^1y_1y^2y_2\cdots y^{d-2}y_{d-1}y^{d-1}y_d$ is a cycle of length $2d$.
For $d=2$, $x$ is replaced by three independent vertices $y^{1}, y^{2}, y^{3}$, the edge $xy_1$ is replaced by edges $y_1y^{1}$, $y_1y^{2}$ and $y^2y^3$,  and the edge $xy_2$ is replaced by edges $y_2y^{1}$ and $y_2y^{3}$. In this case, $C_x$ is the $5$-cycle $y^1y_1y^2y^3y_2$.
Note that the girth of the graph obtained from $G$ by splitting $x$ is also at least five.

\begin{proof}[Proof of Theorem~\ref{X-girth5}]
Let $\lambda$ be the constant from Theorem~\ref{Ckfraccri}, and let $s=4\lambda k+5$.

Let $G$ be a plane graph of girth at least five and let $X$ be a set of vertices of $G$
of degree at most $k$, such that the distance between vertices of $X$ is at least $s$.
Let $X'\subseteq X$ consist of all vertices in $X$ of degree at least two.

For each $x\in X'$, by Theorem~\ref{x-tri-free} there exists an $\{x\}$-enhanced coloring $\psi_x$ of $G$.
Let $G'$ be the graph obtained from $G$ by splitting every vertex in $X'$.
Let $\varphi$ be a $(6:2)$-coloring of $S=\bigcup_{x\in X'} C_x$ defined as follows.
For each $x\in X'$ and each vertex $y\in V(C_x)$ corresponding to a neighbor of $x$ in $G$,
we let $\varphi(y)=\psi_x(y)$.  To other vertices of $S$, we extend the coloring arbitrarily (this is possible,
since they have degree two).

We claim that $\varphi$ extends to a $(6:2)$-coloring of $G'$;
suppose for a contradiction this is not the case.
Let $G''$ be a minimal subgraph of $G'$ such that $S\subset G''$ and $\varphi$ does not
extend to a $(6:2)$-coloring of $G''$.  Clearly, $G''\neq S$; let $G''_0$ be a connected component
of $G''$ such that $E(G''_0)\not\subseteq E(S)$, let $S''=S\cap G''_0$, and let $\varphi''$ be
the restriction of $\varphi$ to $S''$.
By the minimality of $G''$, observe that $G''_0$ is $S''$-critical for $(6:2)$-coloring and
$\varphi''$ does not extend to a $(6:2)$-coloring of $G''_0$.  Let $X''=\{x\in X':C_x\subseteq S''\}$.
If $|X''|\le 1$, and thus $X''\subseteq\{x\}$ for some $x\in X'$,
then note that $\psi_x$ would give an extension of $\varphi''$ to a $(6:2)$-coloring of $G''_0$, which is a contradiction.
Therefore, $|X''|\ge 2$.  By Theorem~\ref{Ckfraccri}, we have $|V(G''_0)|\leq \lambda\sum_{x\in X''}|C_x|\leq 2K\lambda|X''|$.

On the other hand, for each $x\in X''$, let $N_x$ denote the set of vertices of $G''_0$ at distance at
most $(s-3)/2=2K\lambda+1$ from $C_x$.  Since the distance between vertices of $X$ in $G$ is at least $s$,
the distance between $C_x$ and $C_{x'}$ in $G''_0$ for distinct $x,x'\in X''$ is at least $s-2$,
and thus $N_x\cap N_{x'}=\emptyset$.  Furthermore, since $G''_0$ is connected and $|X''|\ge 2$,
$N_x$ contains at least $2K\lambda+1$ vertices on a path from $x$ to $C_{x'}$.
Consequently, $|V(G''_0)|\ge \sum_{x\in X''}|N_x|\ge (2K\lambda+1)|X''|$, which is a contradiction.

Therefore, $\varphi$ indeed extends to a $(6:2)$-coloring of $G'$.
Then the restriction of $\varphi$ to $G-X$ extends to an $X$-enhanced coloring of $G$
(for each $x\in X'$ we set $\varphi(x)=\psi_x(x)$, and for each $x\in X\setminus X'$
we choose $\varphi(x)$ as a $3$-element subset of $\{1,\ldots,6\}$ disjoint from the color set of the neighbor of $x$,
if any).
\end{proof}

\section{Fractional coloring of planar graphs of girth at least 5}

We are now ready to prove our main result.

\begin{proof}[Proof of Theorem~\ref{frac-girth5}]
Let $s$ be the constant of Theorem~\ref{X-girth5} for $k=\Delta$, and let $M_\Delta=\Delta^s$.

Let $G$ be a planar graph of girth at least five with maximum degree at most $\Delta$.
Let $G'$ be the graph obtained from $G$ by adding edges between all pairs
of vertices at distance at most $s-1$. The maximum degree of $G'$ is less
than $\Delta^s$$=M_\Delta$, and thus $G'$ has a coloring by at most $M_\Delta$ colors.
Let $V_1,V_2 ,\ldots,V_{M_\Delta}$ be the color classes of this coloring.
Then the distance in $G$ between any two vertices of the same color class is at least $s$.
By Theorem~\ref{X-girth5}, for $i\in \{1,\ldots,M_\Delta\}$, $G$ has a $V_i$-enhanced set coloring $\varphi_i$
by subsets of $\{6i-5,6i-4,6i-3,6i-2,6i-1,6i\}$. Then $\varphi=\bigcup_{i=1}^{M_\Delta}\varphi_i$ is a set coloring of $G$
by subsets of $\{1,\ldots,6M_\Delta\}$ such that $|\varphi(v)|\geq 2(M_\Delta-1)+3=2M_\Delta+1$ for every $v\in V(G)$.
Therefore, $G$ has a $(6M_\Delta:2M_\Delta+1)$-coloring, and $\chi_f(G)\leq \frac{6M_\Delta}{2M_\Delta+1}$.
\end{proof}

\bibliographystyle{siam}
\bibliography{../data}

\begin{thebibliography}{1}

\bibitem{listfraccyl}
{\sc Z.~Dvo{\v{r}}{\'a}k and X.~Hu}, {\em $(3a:a)$-list-colorability of
  embedded graphs of girth at least five}, ArXiv, 1805.11507 (2018).

\bibitem{dmnichfull}
{\sc Z.~Dvo{\v{r}}{\'a}k and M.~Mnich}, {\em Large independent sets in
  triangle-free planar graphs}, SIAM J. Discrete Math., 31 (2017),
  pp.~1355--1373.

\bibitem{frpltr}
{\sc Z.~Dvo\v{r}\'ak, J.-S. Sereni, and J.~Volec}, {\em Fractional coloring of
  triangle-free planar graphs}, Electronic Journal of Combinatorics, 22 (2015),
  p.~P4.11.

\bibitem{grotzsch1959}
{\sc H.~Gr{\"o}tzsch}, {\em Ein {D}reifarbensatz f\"{u}r dreikreisfreie {N}etze
  auf der {K}ugel}, Math.-Natur. Reihe, 8 (1959), pp.~109--120.

\bibitem{Jon84}
{\sc K.~Jones}, {\em Independence in graphs with maximum degree four}, J.
  Combin. Theory Ser. B, 37 (1984), pp.~254--269.

\bibitem{PosThoHyperb}
{\sc L.~Postle and R.~Thomas}, {\em Hyperbolic families and coloring graphs on
  surfaces}, arXiv, 1609.06749 (2013).

\end{thebibliography}

\end{document}